\documentclass[a4paper, 11pt]{article}
\usepackage{amsmath, amsthm, amssymb, fullpage, graphicx, enumerate}
\usepackage{bbm}
\usepackage[shortlabels]{enumitem}
\usepackage{mathtools}
\usepackage{graphicx}
\usepackage{xcolor}

\usepackage{comment}

\usepackage{enumitem}

\usepackage{centernot}

\usepackage{calrsfs}

\numberwithin{equation}{section}

\theoremstyle{plain}
\newtheorem{theorem}{Theorem}[section]
\newtheorem{lemma}[theorem]{Lemma}
\newtheorem{corollary}[theorem]{Corollary}
\newtheorem{proposition}[theorem]{Proposition}

\theoremstyle{definition}
\newtheorem{definition}[theorem]{Definition}
\theoremstyle{remark}
\newtheorem{remark}[theorem]{Remark}

\newcommand{\labtequ}[2]{
	\begin{equation} \label{#1} 	\begin{minipage}[c]{0.9\textwidth} \center{#2} \end{minipage} \ignorespacesafterend \end{equation} }

\newcommand{\Ex}{\mathbb E}
\renewcommand{\Pr}{\mathbb{P}}

\newcommand{\R}{\ensuremath{\mathbb R}}
\newcommand{\C}{\ensuremath{\mathbb C}}
\newcommand{\Z}{\ensuremath{\mathbb Z}}

\newcommand{\cp}{\mathrm{cap}}

\newcommand{\Lr}[1]{Lemma~\ref{#1}}

\newcommand{\Tr}[1]{Theorem~\ref{#1}}

\DeclareMathAlphabet{\pazocal}{OMS}{zplm}{m}{n}

\begin{document}
	
\title{Analyticity of Gaussian free field percolation observables}

\author{Christoforos Panagiotis\footnotemark[1]\footnote{Universit\'e de Gen\`eve,~christoforos.panagiotis@unige.ch}~ and~Franco Severo\footnotemark[2]\footnote{Universit\'e de Gen\`eve,~franco.severo@unige.ch}
}

\thispagestyle{empty}
\maketitle

\begin{abstract}
	We prove that cluster observables of level-sets of the Gaussian free field on the hypercubic lattice $\Z^d$, $d\geq3$, are analytic on the whole off-critical regime $\R\setminus\{h_*\}$. This result concerns in particular the percolation density function $\theta(h)$ and the (truncated) susceptibility $\chi(h)$.
	As an important step towards the proof, we show the exponential decay in probability for the capacity of a finite cluster for all $h\neq h_*$, which we believe to be a result of independent interest. We also discuss the case of general transient graphs.
\end{abstract}

\section{Introduction}\label{sec:intro}

\textbf{Motivation and main results.} We consider the level-set percolation for the Gaussian free field (GFF) on a connected, locally finite, transient graph $G=(V,E)$. Of particular interest is the case of the hypercubic lattice $\Z^d$ in dimensions $d \geq 3$. The Gaussian free field $\varphi = (\varphi_x)_{x\in V}$ is defined as the centered Gaussian process with covariance
$\Ex(\varphi_x \varphi_y) = g(x,y)$ for all $x,y\in V$,
where $g(\cdot,\cdot)$ stands for the Green function of the simple random walk on $G$. 
Given $h\in \R$, we are interested in the excursion set
$\{\varphi\geq h\}\coloneqq  \{x\in V:\, \varphi_x\geq h\}$ seen as a random subgraph of $G$ (with the induced adjacency).
As $h$ varies, this defines a percolation model for which one may expect to see a phase transition in $h$ from a percolative regime -- where $\{\varphi\geq h\}$ contains an infinite connected component -- to a non-percolative regime -- where all the clusters of $\{\varphi\geq h\}$ are finite. Consider the \emph{percolation density} function defined by
$$\theta(h)\coloneqq \Pr[|\pazocal{C}_o(h)|=\infty],$$
where $\pazocal{C}_o(h)$ denotes the connected component (or cluster) of a fixed origin $o\in V$ in $\{\varphi\geq h\}$.
We can then define the percolation \emph{critical point} $h_*$ given by
$$h_*(G)\coloneqq \sup\{h\in \R:\, \theta(h)>0\}.$$

The first and most fundamental question in percolation theory is the existence of a non-trivial phase transition, which in our case corresponds to $-\infty < h_*< +\infty$. A soft argument due to Bricmont, Lebowitz \& Maes \cite{BricmontLJM87} shows that the GFF percolates above any negative level, i.e.~$h_*(G)\geq 0~(>-\infty)$ for every transient graph $G$ -- it has been recently proved \cite{DrePreRod} that $h_*(\Z^d)>0$ for all $d\geq3$. The opposite inequality $h_*<+\infty$ is more delicate. In the special case $G=\Z^d$, $d\geq3$, this was proved by Rodriguez \& Sznitman \cite{RodriguezSznitman13} -- the case $d=3$ had already been obtained in \cite{BricmontLJM87}.
Other graphs have been proved to satisfy $h_*<+\infty$ \cite{MR3492939}, but this remains open for more general transient graphs. 
Remarkably, this is in contrast with the classical Bernoulli percolation, for which proving the existence of percolative regime is in general much harder than proving the existence of a non-percolative regime -- see \cite{DGRSY20}.

Once the existence of a phase transition is established, the next important question concerns the uniqueness of critical point, i.e.~whether $h_*$ defined above is the only value at which one can see a qualitative change in the large-scale behavior of the model. This immediately raises the question of whether there are critical points at other values of $h$ and how to define them. There are two main approaches to this question. 

From a percolation theory perspective, a natural approach consists in defining alternative critical parameters $\bar{h}$ and $h_{**}$, which characterize a strongly percolative and strongly non-percolative regimes, respectively. In the last decade, this approach has been successfully implemented in the case $G=\Z^d$: definitions appeared in many works --  see e.g.~\cite{RodriguezSznitman13,PopovRath15,MR3390739,MR3417515} -- and more recently it has been proved by Duminil-Copin, Goswami, Rodriguez \& Severo \cite{DGRS20} that indeed $\bar{h}=h_*=h_{**}$. This equality is often referred to as \emph{``sharpness''} of phase transition and is also expected to hold for other transient graphs, but this remains open. The corresponding result for Bernoulli percolation on $\Z^d$ was obtained in the highly influential works of Aizenman \& Barsky \cite{AizBar87} and Menshikov \cite{Men86} (on the subcritical phase) and Grimmett \& Marstrand \cite{GriMar90} (on the supercritical phase).

From the point of view of statistical physics, a classical approach consists in considering a function (such as $\theta$) describing the macroscopic behavior of the model, and define the critical points to be the singularities of that function. Uniqueness of critical point then corresponds to the analyticity of this function on $\R\setminus\{h_*\}$, which is precisely the main result of the present article. Let us mention that the corresponding result for Bernoulli percolation on $\Z^d$ has been proved on the subcritical phase by Kesten \cite{Kes81} and on the supercritical phase by Geogarkopoulos \& Panagiotis \cite{GeoPan1820}. Hermon \& Hutchcroft \cite{HH19} also proved a corresponding result for Bernoulli percolation on non-amenable transitive graphs.

In order to state our main result, we need to introduce some notation. Let $\pazocal{X}$ denote the family of all \emph{finite} subsets of $V$. We say that a \emph{cluster observable} $F:\pazocal{X}\to \C$ has \emph{subexponential growth} if $|F(S)|\leq e^{o(\cp(\overline{S}))}$ as $\cp(\overline{S})\to\infty$. Here $\cp(\overline{S})$ denotes the (harmonic) capacity of $\overline{S}$, the (vertex) closure of $S$ -- see Section~\ref{sec:potential_analytic} for definitions. Finally, for a cluster observable $F:\pazocal{X}\to \C$ and a subset $X\in \pazocal{X}$, consider the function $\overline{F}^X:\R\to \C$ defined by 
$$\overline{F}^X(h)\coloneqq \Ex[F(\pazocal{C}_X(h)) \mathbbm{1}_{|\pazocal{C}_X(h)|<\infty}],$$
where $\pazocal{C}_X(h)$ denotes the union of all clusters in $\{\varphi\geq h\}$ intersecting $X$.

\begin{theorem}\label{thm:main}
	Let $G=\Z^d$, $d\geq 3$. Then for every observable $F:\pazocal{X}\to \C$ of subexponential growth and every $X\in\pazocal{X}$, the function $\overline{F}^X$ well-defined and analytic on $\R\setminus\{h_*\}$.
\end{theorem}

Notice that the analyticity of the percolation density $\theta$ on $\R\setminus\{h_*\}$ follows from Theorem~\ref{thm:main} by taking $F\equiv 1$, for which $\overline{F}^{\{o\}}=1-\theta$. Besides $\theta$, other functions of interest are the \emph{(truncated) susceptibility}
$$\chi(h)\coloneqq \Ex[|\pazocal{C}_o(h)| \mathbbm{1}_{|\pazocal{C}_o(h)|<\infty}],$$
the \emph{(finite) open clusters per vertex}
$$\kappa(h)\coloneqq \Ex[|\pazocal{C}_o(h)|^{-1} \mathbbm{1}_{|\pazocal{C}_o(h)|<\infty}],$$
the \emph{truncated $k$ point function}
$$\tau^{f}_X(h)\coloneqq \Pr[\pazocal{C}_X(h) \text{ connected},~|\pazocal{C}_X(h)|<\infty]$$
and the \emph{(non-truncated) $k$ point function}
$$\tau_X(h)\coloneqq \Pr[\pazocal{C}_X(h) \text{ connected}],$$
where $X\in\pazocal{X}$ with $|X|=k$. The following is a corollary of Theorem~\ref{thm:main}.

\begin{corollary}\label{cor:main}
	Let $G=\Z^d$, $d\geq 3$. Then all the functions $\theta(h)$, $\chi(h)$, $\kappa(h)$, $\tau^{f}_X(h)$ and $\tau_X(h)$, $X\in\pazocal{X}$, are analytic on $\R\setminus\{h_*\}$. 
\end{corollary}

The only function for which Corollary~\ref{cor:main} does not follow readily from Theporem~\ref{thm:main} is the (non-truncated) $k$ point function $\tau_X(h)$. In order to deduce its analyticity, simply notice that by the uniqueness of the infinite cluster (see e.g.~\cite[Remark 1.6]{RodriguezSznitman13})
and the inclusion-exclusion principle, we can write
\begin{align*} 
	\tau_X(h)
	&=\tau^{f}_X(h)+1-\Pr\big[\bigcup_{x\in X}\{|\pazocal{C}_x(h)|<\infty\}\big]\\
	&=\tau^{f}_X(h)+1-\sum_{\emptyset\neq Y\subset X} (-1)^{|Y|+1}\Pr[|\pazocal{C}_Y(h)|<\infty].
\end{align*}
We remark that the analyticity of $\tau_X(h)$ may break down on the supercritical phase if uniqueness of infinite cluster does not hold. Indeed, for Bernoulli percolation there are examples \cite{HH19} of transitive non-amenable graphs for which $\tau$ has a discontinuity at the uniqueness critical point $p_u$, which in this case satisfies $p_c<p_u<1$.

Our proof of analyticity of $\overline{F}^X$ makes crucial use of the following convenient series decomposition. 
For every integer $N\geq 1$ and $h\in\R$, consider the event
$$A^X_N(h)\coloneqq \{|\pazocal{C}_X(h)|<\infty\}\cap\{N-1\leq \cp(\overline{\pazocal{C}_X(h)})<N\}.$$
We can then write
\begin{equation}\label{eq:series}
	\overline{F}^X(h)=\sum_{N=1}^\infty \Ex[F(\pazocal{C}_X(h)) \mathbbm{1}_{A^X_N(h)}].
\end{equation}
With the series \eqref{eq:series} in hands, it is enough to show that each term $\overline{F}^X_N(h)\coloneqq \Ex[F(\pazocal{C}_X(h)) \mathbbm{1}_{A^X_N(h)}]$ can be analytically extended to a domain of $\C$ containing $\R\setminus\{h_*\}$ on which the series converges locally uniformly. A crucial step to establish such a convergence is proving that $\Pr[A^X_N(h)]$ decays (uniformly) exponentially in $N$ for $h\neq h_*$. This is the content of the following theorem, which we believe to be of independent interest.

\begin{theorem}\label{thm:exp_Zd}
	Let $G=\Z^d$, $d\geq 3$. Then for every $\varepsilon>0$ and $X\in\pazocal{X}$, there exists $c=c(X,\varepsilon,d)>0$ such that 
	$\Pr[A^X_N(h)]\leq e^{-cN}$ for every $N\geq 1$ and every $h\in\R$ with $|h-h*|\geq\varepsilon$.
\end{theorem}

It is easy to prove that there exists $c'=c'(d)>0$ such that $\cp(K)\geq c'|K|^{\frac{d-2}{d}}$ for every subset $K\subset \Z^d$. The following corollary thus follows readily from Theorem~\ref{thm:exp_Zd}.

\begin{corollary}\label{cor:vol_decay}
	Let $G=\Z^d$, $d\geq 3$. Then for every $\varepsilon>0$ and $X\in\pazocal{X}$, there exists $c=c(X,\varepsilon,d)>0$ such that $\Pr[N\leq|\pazocal{C}_X(h)|<\infty]\leq \exp\{-cN^{\frac{d-2}{d}}\}$ for every $N\geq 1$ and every $h\in\R$ with $|h-h*|\geq\varepsilon$.
\end{corollary}

The order of exponential decay in the upper bounds provided by Theorem~\ref{thm:exp_Zd} and Corollary~\ref{cor:vol_decay} are believed to be the correct ones. Optimizing on the constant $c$ governing the rate of exponential decay is more challenging and beyond the scope of this article, but we believe that our techniques might shed some light on this problem as well.

In recent years, large deviation problems for GFF percolation events has attracted considerable attention -- see e.g.~\cite{MR3417515,nitzschner2017solidification,nitzschner2018,sznitman2018macroscopic,GRS21}. A common feature in these problems is a deep connection with potential theory and in particular the notion of capacity. Typically, the exponential rate of decay is given by the solution of a constrained optimization problem involving the Dirichlet energy and, in some cases, the percolation density $\theta$ as well \cite{sznitman_bulk_2019,sznitman_2020,sznitman_2021}. It is therefore relevant to understand the regularity of $\theta$ in order to study these optimization problems. Motivated by this, it has been recently proved \cite{sznitmanC1_2019} that $\theta$ is $C^{1}$ for the closely related model of random interlacements. We expect that the techniques developed in the present article may be helpful to study similar questions for the random interlacements and other strongly correlated models as well.
The proof of Theorem~\ref{thm:exp_Zd} is based on a coarse graining argument which is very much in the spirit of the works cited above. However, we would like to highlight a key new aspect of our work: we use a coarse graining procedure that involves, at the same time, multiple scales instead of only one. We describe this multi-scale coarse graining scheme in more details in the end of this section.

We now discuss the case of general transient graphs. First, we observe that the (uniform) exponential decay for $\Pr[A^X_N(h)]$ always implies the analyticity of $\overline{F}^X$ -- see Proposition~\ref{prop:analytic}. By a simple shift-argument, one can show that such exponential decay holds for all negative values of $h$ on any transient graph -- see Proposition~\ref{prop:exp_h<0}. This implies the following theorem. Recall that $h_*(G)\geq0$ is known to hold \cite{BricmontLJM87} for every transient graph $G$.

\begin{theorem}\label{thm:analytic_h<0}
	For every transient graph $G$, every observable $F:\pazocal{X}\to \C$ of subexponential growth and every $X\in\pazocal{X}$, the function $\overline{F}^X$ is well-defined and analytic on $(-\infty,0)$.
\end{theorem}

Under weaker assumptions on the decay of $\Pr[A^X_N(h)]$, we can prove that $\overline{F}^X$ is smooth for observables $F:\pazocal{X}\to \C$ of (at most) polynomial growth, i.e.~satisfying $|F(S)|\leq C|\cp(\overline{S})|^C$ for all $S\in\pazocal{X}$ and some constant $C\in(0,\infty)$. We say that a sequence $(c_N)_{N\geq1}$ decays super-polynomially fast if $\lim_{N\to\infty} \tfrac{\log(c_N)}{\log N}=-\infty$.
We define
$$\widetilde{h}\coloneqq \sup \left\{h\in (-\infty,h_*) \middle\vert~
\begin{array}{l}
	\text{for every $X\in\pazocal{X}$ there exists $(c_N)_{N\geq1}$}\\ \text{decaying super-polynomially fast such} \\
	\text{that $\Pr[A^X_N(h')]\leq c_N$ for every $h'\leq h$}
\end{array}\right\}.$$
We also define an analogous parameter in the subcritical phase
$$\hat{h}\coloneqq \inf \left\{h\in (h_*,+\infty) \middle\vert~
\begin{array}{l}
	\text{for every $X\in\pazocal{X}$ there exists $(c_N)_{N\geq1}$}\\ \text{decaying super-polynomially fast such} \\
	\text{that $\Pr[A^X_N(h')]\leq c_N$ for every $h'\leq h$}
\end{array}\right\}.$$

\begin{theorem}\label{thm:smooth}
	For every transient graph $G$, every observable $F:\pazocal{X}\to \C$ of (at most) polynomial growth and every $X\in\pazocal{X}$, $\overline{F}^X$ is well-defined and $C^{\infty}$ on $\R\setminus[\widetilde{h},\hat{h}]$.
\end{theorem}

The parameters $\widetilde{h}$ and $\hat{h}$ defined above can be seen, respectively, as an alternative definition of the classical parameters $\bar{h}$ and $h_{**}$ mentioned above. Indeed, for the case $G=\Z^d$, it is not hard to prove that $\bar{h}\leq \widetilde{h}\leq h_*\leq \hat{h}\leq h_{**}$, which in turn implies $\widetilde{h}=h_*=\hat{h}$ as the equality $\bar{h}=h_{**}$ is known in this case \cite{DGRS20}. It is natural to expect that the equality $\widetilde{h}=h_*=\hat{h}$ holds in great generality, but sharpness of phase transition remains open beyond $\Z^d$. It is also natural to expect that, independently of sharpness, one might be able to bootstrap the decay of $\Pr[A^X_N(h)]$ from super-polynomial to exponential via a coarse graining argument, thus proving that $\theta$ is analytic on $\R\setminus [\widetilde{h},\hat{h}]$. This is essentially what we do in the proof of Theorem~\ref{thm:exp_Zd} for the case $G=\Z^d$: we start from a sub-optimal decay provided by the assumption $h\in\R\setminus[\bar{h},h_{**}]~(=\R\setminus\{h_*\}$ by \cite{DGRS20}) and enhance it to the desired exponential decay through a coarse graining argument -- see the discussion below for more details. On general graphs though, developing a coarse graining argument is more challenging due to a poorer understanding of their geometry.

\quad

\textbf{About the proof.} As mentioned above, our proof makes crucial use of the series \eqref{eq:series}. We first use a shift-argument based on the Cameron--Martin formula to naturally construct an analytic extension of the function $\overline{F}^X_N=\Ex[F(\pazocal{C}_X(h)) \mathbbm{1}_{A^X_N(h)}]$ to the whole complex plane $\C$ for every $N\geq1$. This construction provides a simple way to effectively estimate the growth of this entire function along the imaginary direction. More precisely, we prove in Proposition~\ref{prop:extension_bound} that $\overline{F}^X_N(h+it)\leq \exp\{\tfrac{1}{2}t^2N\}\overline{|F|}^X_N(h)$ for every $N\geq1$ and $h,t\in\R$. Due to this result, it is not difficult to deduce the locally uniform convergence (and therefore analyticity) of the series~\eqref{eq:series} from the (uniform) exponential upper bound for $\Pr[A^X_N(h)]$ on the real line -- see Proposition~\ref{prop:analytic}. This exponential bound is then provided in the case $G=\Z^d$ by Theorem~\ref{thm:exp_Zd}, which is the most technical part of this article. 

Before discussing the ideas involved in the proof of Theorem~\ref{thm:exp_Zd}, we would like to highlight some key differences between GFF level-sets and Bernoulli percolation. Kesten's proof \cite{Kes81} of analyticity for Bernoulli percolation on the subcritical phase is based on a series expansion similar to \eqref{eq:series}, but in terms of the cluster size $N=|\pazocal{C}_X|$. Since in the subcritical phase the cluster size decays exponentially in probability \cite{AizBar87,Men86} and the expansion in the imaginary direction also grows (at most) exponentially in $N$, one can prove that the series converges locally uniformly near the real line and is therefore analytic. This strategy does not work in the supercritical phase though: while the expansion in the imaginary direction is still exponential in $N$, the decay of the cluster probabilities is exponential in its \emph{boundary size} \cite{KesZha90}, which is typically of order $N^{\frac{d-1}{d}}=o(N)$. Motivated by this issue, Georgakopoulos \& Panagiotis \cite{GeoPan1820} considered a series decomposition in terms of the size of \emph{(multi-)interfaces} instead, in which case both the expansion in the imaginary direction and the decay on the real line are of the same exponential order. For the GFF level-sets though, none of these decompositions can work as the decay on the real line is subexponential in both the volume and boundary sizes. 
Nevertheless, we observe that both the imaginary expansion and the decay of cluster probability (in both subcritical and supercritical phases!) are exponential in the \emph{capacity} of the cluster, thus allowing us to make effective use of the series expansion \eqref{eq:series}. 
This fact is due to an \emph{entropic repulsion} phenomenon that emerges from the strong (non-integrable) correlations of the GFF, which in turn are deeply related to the potential theory attached to the random walk.

We will now outline the main ideas present in the proof of Theorem~\ref{thm:exp_Zd}. As mentioned above, a quite substantial multi-scale coarse graining argument takes place in the proof. We start by discussing the more natural single-scale coarse graining approach with the hope of making the need for a multi-scale argument more apparent. This single-scale approach would consist in choosing an appropriate scale $L$ and observe that on the event $A^X_N(h)$ one can find a family $\pazocal{F}$ of $L$-boxes on which an unlikely event (so called \emph{bad event}) happens. Then one can hope to prove that, for every given $\pazocal{F}$, the probability that all of these boxes are bad is at most $e^{-cN}$, while keeping the combinatorial complexity (i.e.~the number of possible families $\pazocal{F}$) of order $e^{o(N)}$. In order to prove the desired exponential upper bound, one can use the harmonic decomposition of GFF on each box of $\pazocal{F}$ into the sum of a local and a global field and then consider two cases: either most boxes of $\pazocal{F}$ are \emph{globally bad} -- which corresponds to the global (harmonic) field deviating from $0$ -- or many boxes are \emph{locally bad} -- which corresponds to the occurrence of an unlikely percolation event for the local field. By applying a large deviation result of Sznitman \cite{MR3417515}, one can prove that the probability of the first case decays exponentially in the capacity, i.e.~it is smaller than $e^{-cN}$, as desired. In the second case though, one can use independence to show that its probability is smaller than $p_L^{|\pazocal{F}|}$, where $p_L$ is the probability of a single $L$-box being locally bad. On the one hand, since the available a priori bound on $p_L$ is only stretched exponential in $L$ and the geometry of $\pazocal{F}$ is completely arbitrary, one quickly notices that in order for the desired inequality $p_L^{|\pazocal{F}|}\leq e^{-cN}$ to hold uniformly in $\pazocal{F}$, it is necessary to choose $L$ not too large. On the other hand, because of the arbitrary geometry of $\pazocal{F}$ again, it is necessary to take $L$ sufficiently large in order to have a combinatorial complexity of order $e^{o(N)}$. As a consequence, choosing such a scale $L$ becomes impossible, suggesting the need of a multi-scale approach.

Our multi-scale coarse graining construction goes roughly as follows. For each configuration $\varphi\in A_N(h)$ we construct a set of bad (and very bad) boxes $\pazocal{F}$ consisting of multiple scales. We do so inductively in the scales, starting by a sufficiently large scale $L$ such that the combinatorial complexity is of order $e^{o(N)}$. We then look at all the boxes where something unlikely happens -- these boxes are called \emph{bad} -- and we add to $\pazocal{F}$ all those boxes where something ``very unlikely'' happens -- these boxes are called \emph{very-bad}. Here ``very unlikely'' corresponds to an event for which an improved a priori upper bound of type $q_L\leq e^{-c\,\cp(B)}$ holds. If these boxes have capacity of order $N$, we are done. Otherwise, we can go down to a smaller scale $L'<L$ and inspect the bad $L'$-boxes contained in the remaining $L$-boxes (i.e.~bad but not very-bad) and add to $\pazocal{F}$ those $L'$-boxes which are very-bad. By continuing this process, we eventually obtain either a family of very-bad boxes with capacity of order $N$ or a very large number of bad boxes of the smallest scale $L_0$. We can then prove that the probability of both cases is smaller than $e^{-cN}$. Since each time we go down one scale we look only inside certain boxes of the previous scale, it turns out that we can do so by keeping the combinatorial complexity of order $e^{o(N)}$, as desired. For this construction to work though, one has to define the notions of bad and very-bad boxes in a very careful way so that a certain \emph{propagation property} holds -- see item (iii) of Definition~\ref{def:bad_events}.

\quad

\textbf{Organization of the paper.} In Section~\ref{sec:potential_analytic} we review the potential theory attached to the simple random walk and describe the shift-argument used to extend each term of the series \eqref{eq:series} to an entire function. We then prove Theorems~\ref{thm:analytic_h<0} and \ref{thm:smooth} and also deduce Theorem~\ref{thm:main} from Theorem~\ref{thm:exp_Zd}, to which the remaining sections are dedicated. In Section~\ref{sec:exp_decay_Zd}, we describe the large deviation argument used to prove Theorem~\ref{thm:exp_Zd}. In Section~\ref{sec:proof_coarse_graining} we prove the (deterministic) multi-scale coarse graining theorem stated in Section~\ref{sec:exp_decay_Zd}. Finally, in Sections~\ref{sec:decay_bad} and \ref{sec:decay_verybad} we prove the decay in probability for the notions of bad and very-bad boxes introduced in Section~\ref{sec:exp_decay_Zd}. 

\quad

\textbf{Acknowledgments.} The second author would like to thank Subhajit Goswami, Alexis Pr\'{e}vost and Pierre-Fran\c{c}ois Rodriguez for inspiring discussions on large deviations for GFF level-sets. This research was supported by the Swiss National Science Foundation and the NCCR SwissMAP.

\section{Potential theory and analytic extension}\label{sec:potential_analytic}

We start by introducing some notation. For any pair $x,y\in V$ we write $x\sim y$ if $\{x,y\}\in E$. Given $S\in\pazocal{X}$, we may consider its (inner) boundary $\partial S\coloneqq \{x\in S:~\exists y\in V\setminus S,~x\sim y\}$, its outer boundary $\partial^{out} S\coloneqq \{x\in V \setminus S:~\exists y\in S,~x\sim y\}$ and its closure $\overline{S}\coloneqq S\cup \partial^{out}S$.

We now recall some potential theory attached to simple random walk (SRW) on the graph $G=(V,E)$, which is assumed to be locally finite, connected and transient for the SRW.
We denote by $P_x$ the canonical law of the discrete-time SRW on 
$G$ starting at $x \in V$ and write $(X_n)_{n \geq 0}$ for the corresponding 
process. We let $g(\cdot,\cdot)$ stand for the Green function of the walk,
\begin{equation}\label{green}
	g(x,y) \coloneqq \dfrac{1}{d(y)}\sum_{n=0}^{\infty} P_x [X_n = y], \quad \text{ for }x,y \in V,
\end{equation}
where $d(y)$ denotes the degree of $y$. It is well known that the Green function is finite (as $G$ is transient), symmetric and positive-definite. Therefore, we can effectively define the GFF $\varphi=(\varphi_x)_{x\in V}$ as the centered Gaussian field with covariance matrix $g$. In the case of $\Z^d$, $d\geq3$, it is well known that $g(x,y)\asymp \Vert x-y \Vert^{2-d}$.

Given $K \subset V$ and $x\in V$, we consider the \emph{equilibrium measure} $e_K(x) \coloneqq  d(x)P_x [\widetilde{H}_K = \infty] \,1_{x\in K}$, where $\widetilde{H}_K\coloneqq \min\{n\geq 1 :~X_n\in K\}$. The \emph{capacity} of $K$ is defined as its total mass,
\begin{equation}\label{capacity}
	\cp(K) \coloneqq  \sum_{x \in K} e_K(x).
\end{equation}
The capacity is an increasing and sub-additive function, i.e.~$\cp(A)\leq\cp(A\cup B)\leq \cp(A)+\cp(B)$ for every $A,B\subset V$.
The following variational characterization of the capacity is useful for obtaining lower bounds:
\begin{equation}\label{eq:variational_meas}
	\cp(K)=\big(\inf_{\nu} E(\nu)\big)^{-1}, 
\end{equation}
where $E(\nu)\coloneqq \sum_{x,y} \nu(x)\nu(y)g(x,y)$ and the infimum ranges over all probability measures $\nu$ supported on $K$. As a direct consequence, one has the following inequality
\begin{equation}\label{eq:cap_bound}
	\frac{|K|}{\sup_{x\in K} \sum_{y\in K} g(x,y)}\leq \cp(K)\leq \frac{|K|}{\inf_{x\in K} \sum_{y\in K} g(x,y)}.
\end{equation}
In the special case of the box $B_L=[0,L)^d$ on $\Z^d$, one can conclude that
\begin{equation}\label{ball cap}
	\cp(B_L) \asymp L^{d-2} \quad \text{for all } L\geq 1.
\end{equation}
The optimizing measure in \eqref{eq:variational_meas} is precisely the normalized equilibrium measure $\overline{e}_K(x)\coloneqq e_K(x)/\cp(K)$. 
Further, for every $K \subset K' \subset \subset \Z^d$ one has the sweeping identity
\begin{equation}\label{sweeping}
	\cp(K) = \sum_{x\in K'} e_{K'}(x) \Pr_x[H_K < \infty],
\end{equation}
where $H_K\coloneqq \min\{n\geq 0 :~X_n\in K\}$. 
Consider the Dirichlet inner product defined as
\begin{equation}\label{eq:Dirichlet}
	\pazocal{E}(f,g) \coloneqq -\sum_{x \in V} \Delta f(x) g(x)
\end{equation}
for every pair of function $f,g: V \rightarrow \R$ for which the sum converges (for instance, if either $\Delta f$ or $g$ has finite support), where $\Delta f(x)\coloneqq \sum_{y\sim x}(f(y)-f(x))$ is the Laplacian of $f$. 
One also has the following variational characterization of capacity in terms of the Dirichlet energy
\begin{equation}\label{eq:variational_func}
	\cp(K)=\inf_{f} \pazocal{E}(f,f),
\end{equation}
where the infimum ranges over all functions $f$ such that $\pazocal{E}(f,f)$ is well defined and $f(x)\geq1$ for every $x\in K$.
The optimizing function in \eqref{eq:variational_func} is called the \emph{harmonic potential} of $K$ and is given by 
$$f_K(x)\coloneqq P_x[H_K<\infty].$$ 
In fact, $f_K$ takes value $1$ on $K$ and is harmonic on $V\setminus K$,~i.e.~$\Delta f_K(x)=0$ for all $x\in V\setminus K$. 

Given a function $f:V \rightarrow \C$ for which the Laplacian $\Delta f(x)$ has finite support, we introduce the complex measure
\begin{equation}\label{def tilde P}
	\widetilde{\Pr}_f(d\varphi)\coloneqq\exp\left\{-\tfrac{1}{2}\pazocal{E}(f,f)-\pazocal{E}(f,\varphi)\right\}\Pr(d\varphi).
\end{equation}
Notice that when $f$ takes real values, the Cameron--Martin formula implies that $\widetilde{\Pr}_{f}$ is a probability measure and furthermore, the law of $\varphi$ under $\widetilde{\Pr}_{f}$ coincides with the law of $\big(\varphi_x-f(x)\big)$ under $\Pr$. This observation will allow us to extend the probability of \textit{local} events to the complex plane.

\begin{proposition}\label{prop:extension_bound}
	For every $X\in \pazocal{X}$ and $N\geq1$, the function $\overline{F}^X_N(h)=\Ex[F(\pazocal{C}_X(h)) \mathbbm{1}_{A^X_N(h)}]$ extends to an entire function such that for every $h,t\in \R$,
	\begin{equation}\label{eq:F_N_bound}
		|\overline{F}^X_N(h+it)|\leq \exp\left\{\tfrac{1}{2}t^2N\right\}\overline{|F|}^X_N(h),
	\end{equation}
	where $\overline{|F|}^X_N(h)=\Ex[|F(\pazocal{C}_X(h))| \mathbbm{1}_{A^X_N(h)}]$.
\end{proposition}
\begin{proof}
	Let $S\in\pazocal{X}$. We start by extending $h\mapsto\Pr[\pazocal{C}_X(h)=S]$ to the complex plane. For every $z\in\C$, we define 
	\begin{align}\label{eq:theta_S}
		\begin{split}
			\theta^X_S(z)\coloneqq \widetilde{\Pr}_{zf_{\overline{S}}}[\pazocal{C}_X(0)=S]&= \Ex[\exp\left\{-\tfrac{1}{2}z^2\pazocal{E}(f_{\overline{S}},f_{\overline{S}})-z\pazocal{E}(f_{\overline{S}},\varphi)\right\}\mathbbm{1}_{\pazocal{C}_X(0)=S}]\\
			&=\exp\{-\tfrac{1}{2}z^2\cp(\overline{S})\}\sum_{k=0}^{\infty} \dfrac{\Ex[\left(-\pazocal{E}(f_{\overline{S}},\varphi)\right)^k \mathbbm{1}_{\pazocal{C}_X(0)=S}]}{k!}z^k.	
		\end{split}
	\end{align}
	First notice that since the event $\{\pazocal{C}_X(0)=S\}$ only depends on $\varphi$ restricted to $\overline{S}$ and $hf_{\overline{S}}=h$ on $\overline{S}$, it follows from the Cameron--Martin formula that $\theta^X_{S}(h)$ is indeed equal to $\Pr[\pazocal{C}_X(h)=S]$ for $h\in\R$. 
	In order to prove that $\theta^X_S(z)$ is analytic on $\C$ it suffices to show that the series in \eqref{eq:theta_S} converges locally uniformly. Indeed, this follows directly from the fact that for all $n\geq0$,
	$$\left\lvert\sum_{k=0}^{n} \dfrac{\left(-\pazocal{E}(f_{\overline{S}},\varphi)\right)^k}{k!}z^k \right\rvert \leq \sum_{k=0}^\infty \dfrac{\left\lvert\pazocal{E}(f_{\overline{S}},\varphi)\right\rvert ^k}{k!} |z|^k=\exp\left\{\lvert z\pazocal{E}(f_{\overline{S}},\varphi)\rvert\right\},$$
	and $\Ex[\exp\left\{\lvert z\pazocal{E}(f_{\overline{S}},\varphi)\rvert \right\}]$ is finite for every $z\in\C$ as $\pazocal{E}(f_{\overline{S}},\varphi)$ is a Gaussian random variable. 
	
	We will now obtain a bound for $\theta^X_S(h+it)$ in terms of $\theta^X_S(h)$ for $h,t\in \R$. By the Cameron--Martin formula, we have
	$$\theta^X_S(h+it)=\Pr_{itf_{\overline{S}}}[\pazocal{C}_X(h)=S]= \Ex[\exp\left\{\tfrac{1}{2}t^2\cp(\overline{S})-it\pazocal{E}(f_{\overline{S}},\varphi)\right\}\mathbbm{1}_{\pazocal{C}_X(h)=S}].$$
	Since $\lvert\exp\left\{-iy\pazocal{E}(f_{\overline{S}},\varphi)\right\}\rvert=1$ a.s., we obtain
	\begin{equation}\label{eq:theta_S_bound}
		|\theta^X_S(h+it)|\leq \exp\{\tfrac{1}{2}t^2\cp(\overline{S})\}\Pr[\pazocal{C}_X(h)=S].
	\end{equation}
	Finally, for every $z\in \C$, we define
	\begin{equation} \label{eq:F_bar_def}
		\overline{F}^X_N(z)\coloneqq \sum_{S\in \pazocal{A}_N} F(S)\theta^X_S(z).
	\end{equation}
	Here $\pazocal{A}_N$ denotes the family of all sets $S\in\pazocal{X}$ such that $N-1\leq \cp(\overline{S})<N$. By \eqref{eq:theta_S_bound}, 
	$$\sum_{S\in A_N}\lvert F(S)\theta^X_S(z)\rvert\leq \exp\{\tfrac{1}{2}\lvert \text{Im}(z)\rvert^2 N\}\sup_{S\in A_N} |F(S)| <\infty.$$
	We can then apply the Weierstrass M-test to conclude that the series in \eqref{eq:F_bar_def} converges locally uniformly and therefore $\overline{F}^X_N$ is indeed analytic on $\C$. The inequality \eqref{eq:F_N_bound} follows readily from \eqref{eq:theta_S_bound}.
\end{proof}

With Proposition~\ref{prop:extension_bound} in hands, we can now easily obtain a sufficient condition for the analyticity of $\overline{F}^X$.

\begin{proposition}\label{prop:analytic}
	Let $X\in\pazocal{X}$. If there exists a constant $t>0$ such that $\Pr[A^X_N(h)]\leq e^{-tN}$ for every $N\geq 1$ and $h\in(a,b)$, then $\overline{F}^X$ is analytic on $(a,b)$ for every observable $F$ of subexponential growth.
\end{proposition}
\begin{proof}
	By Proposition \ref{prop:extension_bound} and our assumption on the decay of $\Pr[A^X_N(h)]$, we obtain that $\lvert \overline{F}^X_N(z)\rvert\leq \exp\{-\tfrac{1}{2}tN\}\sup_{S\in \pazocal{A}_N} \lvert F(S)\rvert$ for every $z\in (a,b)\times (-\sqrt{t},\sqrt{t})$.  
	By the subexponential growth of $F$ it follows that the series $\overline{F}^X(z)=\sum_{N=1}^\infty \overline{F}^X_N(z)$ converges uniformly on $(a,b)\times (-\sqrt{t},\sqrt{t})$, hence it is analytic on that set.
\end{proof}

Notice that Theorem~\ref{thm:main} follows directly from Proposition~\ref{prop:analytic} and Theorem~\ref{thm:exp_Zd}, whose proof is presented in the following sections. Theorem~\ref{thm:analytic_h<0} follows from Proposition~\ref{prop:analytic} and the following simple result. Recall that $\{\varphi\geq h\}$ is known to percolate for every $h<0$ on any transient graph \cite{BricmontLJM87}.

\begin{proposition}\label{prop:exp_h<0}
	For every transient graph $G$, $X\in\pazocal{X}$, $h<0$ and $N\geq 1$, we have $$\Pr[A^X_N(h)]\leq \exp\{-\tfrac{1}{2}h^2(N-1)\}.$$
\end{proposition}

\begin{proof}
	Let $h<0$ and $S\in\pazocal{A}_N$. Recall that by the Cameron--Martin formula, $$\Pr[\pazocal{C}_X(h)=S]=\widetilde{\Pr}_{hf_{\overline{S}}}[\pazocal{C}_X(0)=S]=\exp\left\{-\tfrac{1}{2}h^2\pazocal{E}(f_{\overline{S}},f_{\overline{S}})\right\}\Ex[\exp\left\{-h\pazocal{E}(f_{\overline{S}},\varphi)\right\}1_{\pazocal{C}_X(0)=S}].$$ Notice that on the event $\{\pazocal{C}_X(0)=S\}$, we have $\varphi_x\leq 0$ for every $x\in \partial^{out} S\supset \partial \overline{S}$. Moreover, $\Delta f_{\overline{S}}(x)=0$ for every $x\in V\setminus \partial \overline{S}$ and $\Delta f_{\overline{S}}(x)\leq 0$ for every $x\in \partial \overline{S}$. It follows that $\exp\{-h\pazocal{E}(f_{\overline{S}},\varphi)\}\leq 1$ on the event $\{\pazocal{C}_X(0)=S\}$. Furthermore, $\pazocal{E}(f_{\overline{S}},f_{\overline{S}})=\cp(\overline{S})\geq N-1$. Overall we obtain $$\Pr[\pazocal{C}_X(h)=S]\leq \exp\left\{-\tfrac{1}{2}h^2 (N-1)\right\}\Pr[\pazocal{C}_X(0)=S]$$
	and the desired inequality follows by summing over $S$.
\end{proof}

We finish this section by proving Theorem~\ref{thm:smooth}.

\begin{proof}[Proof of \Tr{thm:smooth}]
	Let us write $D(h,R)$ for the closed disk in the complex plane that is centred at $h$ and has radius $R$. Consider some $h\in\R\setminus[\tilde{h},\hat{h}]$ and let $R=N^{-1/2}$. By Proposition~\ref{prop:extension_bound} and the Cauchy estimate, we can bound the $k$th derivative of $\overline{F}^X_N$ as follows
	$$\lvert\partial^k\overline{F}^X_N(h)\rvert\leq\dfrac{k!M_R}{R^k},$$
	where $M_R=\sup_{z\in D(h,R)} |\overline{F}^X_N(z)|$. The inequality \eqref{eq:F_N_bound} implies that 
	$$M_R\leq e^{1/2} \sup_{S\in\pazocal{A}_N} \lvert F(S)\rvert\sup_{h'\in [h-R,h+R]} \Pr[A^X_N(h')].$$
	Thus, by the subpolynomial growth of $F$ and the superpolynomial decay of $\Pr[A^X_N(h')]$, it follows that $\lvert\partial^k \overline{F}^X_N(h)\rvert$ decays to $0$ super-polynomially fast and uniformly on compact subsets of $\R\setminus[\tilde{h},\hat{h}]$. We can now conclude that the sum $\sum_{N=1}^{\infty} \lvert\partial^k \overline{F}^X_N(h)\rvert$ converges uniformly on compact subsets of $\R\setminus[\tilde{h},\hat{h}]$, hence the $k$th derivative of $\overline{F}^X(h)$ exists and is equal to $\sum_{N=1}^{\infty} \partial^k \overline{F}^X_N(h)$.
\end{proof}
	
\section{Exponential decay of capacity on $\Z^d$}\label{sec:exp_decay_Zd}

In this section, we will introduce some definitions and state the technical results needed for the proof of \Tr{thm:exp_Zd}. Since $\Z^d$ is transitive and the capacity is sub-additive, by a union bound we can assume without loss of generality that $X=\{o\}$ and henceforth omit $X$ from the notation.

\subsection{Markov decomposition and harmonic deviations}\label{subsec:decomp_deviation}

We start by introducing some notation. Given $L\geq 1$, let $B_L=[0,L)^d$, $U_L=[-L,2L)^d$, $D_L=[-3L,4L)^d$ and $K_L=[-100L,101L)^d$. We will write $B_L(z)=z+B_L$, $U_L(z)=z+U_L$, $D_L(z)=z+D_L$ and $K_L(z)=z+K_L$ for their translates with respect to a vertex $z\in L\Z^d$. We will view $L\Z^d$ both as a graph that is naturally isomorphic to $\Z^d$ and as the collection of all the boxes $B_L(z)$.
Given a box $B=B_L(z)$, we consider the Gaussian fields
\begin{equation}\label{eq:decomp}
	\xi^B_x\coloneqq  E_x\big[\varphi_{X_{T_{K}}}\big] =\sum_{y} P_x[X_{T_K}=y]\varphi_y, \quad  \psi^B_x \coloneqq \varphi_x- \xi_x^B, \quad  \text{ for }x \in \mathbb{Z}^d,
\end{equation}
where $K=K_L(z)$ and $T_K\coloneqq \min\{n\geq0:~X_n\notin K\}$. 
One then has the decomposition $$\varphi_x=\psi^B_x+\xi^B_x, \quad \forall x\in \Z^d.$$ 
It is clear that $\xi^B_x=\varphi_x$ (and therefore $\psi^B_x=0$) for every $x\in \Z^d\setminus K$. The Markov property implies that $\psi^B$ is independent of $\sigma(\varphi_x, x\in \Z^d\setminus K)$, hence it is independent from $\xi^B$.
Moreover, $\xi^B$ is harmonic in $K$ and the covariance matrix of $\psi^B$ is equal to the Green function $g_K$ for simple random walk killed on the boundary of $K$. The fields $\xi^B$ and $\psi^B$ are often called harmonic and local fields, respectively.
The aforementioned decomposition of $\varphi$ is of great importance for large deviation results as it will allow us to distinguish local contributions (driven by $\psi$) from global ones (driven by $\xi$). In this subsection we focus on estimating the global contributions, which correspond to deviations of $\xi$ and are governed by the capacity.

Let $\varepsilon>0$ and $L\geq 1$. We say that the box $B=B_L(z)\in L\Z^d$ is $(\xi,\varepsilon)$-\emph{good} if 
$$\lvert\xi^B_x\rvert <\varepsilon \quad \text{for every } x\in D,$$
where $D=D_L(z)$.
If $B$ is not $(\xi,\varepsilon)$-good, we will call it $(\xi,\varepsilon)$-\emph{bad}. Sznitman \cite{MR3417515} obtained a precise estimate for the probability that many boxes of the same scale are $(\xi,\varepsilon)$-bad. For our purposes, a multi-scale version of Sznitamn's result is necessary. To formally state this new version, we will need the following definition. Consider a family $\pazocal{F}$ of boxes of $L_1\Z^d,L_2\Z^d,\ldots,L_r\Z^d$, where $L_1<L_2<\ldots<L_r$ are integers. We say that $\pazocal{F}$ is \emph{well-separated} if for every pair of boxes $B_{L_i}(z), B_{L_j}(w)\in \pazocal{F}$, the boxes $K_{L_i}(z)$ and $K_{L_j}(w)$ are disjoint. We remark that for a well-separated family $\pazocal{F}$, the local fields $\psi^{B}$, $B\in\pazocal{F}$, are independent from each other, which will be useful in the following sections in estimating the probability of certain events. Finally, we define $$\Sigma=\Sigma(\pazocal{F})\coloneqq \bigcup_{B\in \pazocal{F}} B.$$ 
The following is a slight modification of Sznitman's result.

\begin{lemma}\label{lem:multi_Szn}
	There is a constant $c_0>0$ such that the following holds. For every $\varepsilon>0$ there is a constant $\delta>0$ such that for every well-separated collection $\pazocal{F}$ with $|\pazocal{F}|\leq \delta \cp(\Sigma)$, we have
	$$\Pr(\text{$B$ is $(\xi,\varepsilon)$-bad} \; \forall B\in \pazocal{F})\leq \exp\left(-c_0\varepsilon^2 \cp(\Sigma)\right).$$
\end{lemma}
\begin{proof}
	It suffices to prove that for some constant $c'>0$, we have
	\begin{equation}\label{eq:harm_dev_proof1}
	\Pr[\bigcap_{B\in \pazocal{F}}\{\sup_{D} \xi^B\geq \varepsilon\}]\leq \exp\left(-c'\varepsilon^2 \cp(\Sigma)\right).
	\end{equation}
	Indeed, notice that $\xi^B$ are centered and either $\pazocal{F}^-=\{B\in \pazocal{F}: \; \inf_{D} \xi^B\leq -\varepsilon\}$ or $\pazocal{F}^+=\{B\in \pazocal{F}: \; \sup_{D} \xi^B\geq \varepsilon\}$ has capacity at least $\cp(\Sigma)/2$ by the sub-additivity of the capacity. Moreover, there are $2^{|\pazocal{F}|}\leq 2^{\delta \cp(\Sigma)}$ possibilities for $\pazocal{F}^{\pm}$, so it is enough to take $0<\delta<<c'$.

	The proof of \eqref{eq:harm_dev_proof1} is essentially the same as in \cite[Corollary 4.4]{MR3417515}. We will point out the necessary changes. The results mentioned throughout this proof are from \cite{MR3417515}. 
	We attach to $\pazocal{F}$ the collection $F$ of functions $f$ from $\pazocal{F}$ into $\Z^d$ such that $f(B)\in D$. Let $\lambda(B)=\nu(B)/\cp(\Sigma)$ and define
	$$Z_f=\sum_{B\in\pazocal{F}} \lambda(B)\psi^B(f(B))$$
	and
	$$Z=\sup_{f\in F}Z_f.$$
	We need to show that there exists a constant $C=C(d)>0$ such that 
	\begin{equation}\label{variance}
		\text{var}(Z_f)\leq \dfrac{C}{\cp(\Sigma)}
	\end{equation}
	for every $f\in F$ and 
	\begin{equation}\label{mean}
		\Ex[Z]\leq C \left(\frac{|\pazocal{F}|}{\cp(\Sigma)}\right)^{1/2}.
	\end{equation}
	The first inequality can be obtained by arguing as in the proof of \cite[Theorem 4.2]{MR3417515}. Due to the fact that boxes in $\pazocal{F}$ have in general different scales, we need to slightly modify the argument from \cite[Theorem 4.2]{MR3417515} ion order to obtain the second inequality. Indeed, following the proof of \cite[Lemma 4.3]{MR3417515} we get
	\begin{align}\label{second_moment}
		\begin{split}
		\Ex[&(Z_f-Z_k)^2]\\ &\leq C' \sum_{B,B'\in \pazocal{F}}\lambda(B)\lambda(B') \dfrac{\Vert f(B)-k(B)\Vert_{\infty}\Vert f(B')-k(B')\Vert_{\infty}}{LL'} \Ex[\xi^B(f(B))\xi^{B'}(f(B'))]
		\end{split}
	\end{align}
	for every $f,k\in F$, where $C'$ is a constant, $L$ denotes the scale of $B$ and $L'$ denotes the scale of $B'$.
	It follows from \eqref{second_moment}, \eqref{variance} and the fact that $\Vert f(B)-k(B) \Vert_\infty\leq 7L$ that
	$$\Ex[(Z_f-Z_k)^2]\leq 49C'\Ex(Z_f^2)\leq \dfrac{C''}{\cp(\Sigma)},$$
	where $C''=49CC'$.
	Setting $\widetilde{Z}_f=\sqrt{\cp(\Sigma)}Z_f$ we obtain
	$$\Ex[(\widetilde{Z}_f-\widetilde{Z}_k)^2]^{1/2}\leq \sqrt{C''}.$$
	Now given $\varepsilon\in (0,\sqrt{C''}]$, for every $L\geq 1$ we pick the largest integer $l$ such that $l\leq 7\varepsilon L/\sqrt{C''}$ and for each box $B\in \pazocal{F}$ of scale $L$, we partition $D$ into disjoint boxes, each having $\Vert \cdot \Vert_{\infty}$-diameter at most $l$. If $f,k\in F$ are such that for every $B\in \pazocal{F}$, $f(B)$ and $k(B)$ lie in the same box of $D$, then it follows from \eqref{second_moment} that $\Ex[(Z_f-Z_k)^2]\leq \varepsilon$. Arguing as in page $1820$ of \cite{MR3417515} we obtain \eqref{mean}. We can now use the Borell-TIS inequality as in the proof of \cite[Corollary 4.4]{MR3417515} to obtain 
	$$\Pr[\bigcap_{B\in \pazocal{F}}\{\sup_{D} \xi^B\geq \varepsilon\}]\leq \exp\left\{-\dfrac{1}{2\sigma^2}(\varepsilon-|\Ex(Z)|)_{+}\right\}$$
	with $\sigma^2=\sup_{F} \text{var}(Z_f)$.
	With \eqref{variance} and \eqref{mean} in hands, the desired result follows once we choose $\delta$ so that $C\sqrt{\delta}\leq \varepsilon/2$ and $0<\delta<<c'$.
\end{proof}

Notice that by applying Lemma~\ref{lem:multi_Szn} to a single box $B\in L\Z^d$ and recalling \eqref{ball cap}, we have
\begin{equation}\label{eq:single_bad}
	\Pr(\text{$B$ is $(\xi,\varepsilon)$-bad})\leq \exp\left(-c\varepsilon^2 L^{d-2}\right).
\end{equation}

\subsection{Bad boxes and multi-scale coarse graining}\label{subsec:coarse_graining}

Our aim now is to set up the abstract multi-scale coarse graining scheme used to prove \Tr{thm:exp_Zd}. This is encapsulated in \Tr{thm:multi-int} below, which is purely deterministic and whose proof is postponed to Section~\ref{sec:proof_coarse_graining}. In the next subsections, we deduce the desired exponential decay of capacity in the subcritical and supercritical phases separately by applying \Tr{thm:multi-int} with well chosen notions of ``bad'' and ``very-bad'' events. 

Let us start by giving some definitions and introducing some notations. For every $L\geq 1$ and $h\in \R$, let $\pazocal{C}_o(h,L)$ be the set of boxes of $L\Z^d$ that contain a vertex of $\pazocal{C}_o(h)$. We recall that the inner vertex boundary $\partial \pazocal{C}_o(h,L)$ of $\pazocal{C}_o(h,L)$ is defined as the set of boxes in $\pazocal{C}_o(h,L)$ that have a neighbour in $L\Z^d\setminus \pazocal{C}_o(h,L)$.

We will introduce a general framework that will allow us to study both the supercritical and the subcritical regime. To this end, we consider a family of ``bad events'' indexed by boxes $B=B_L(z)\in L\Z^d$, satisfying certain properties.

\begin{definition}[Admissible bad events]\label{def:bad_events}
	Given $L\geq1$, we say that a family of events $\pazocal{E}^{i}_B$ with $B=B_L(z)\in L\Z^d$, $i\in\{b,vb\}$, is \emph{$h$-admissible} if it satisfies the following properties:
	\begin{enumerate}[(i)]
		\item \label{item:disj} $\pazocal{E}^{b}_B$ and $\pazocal{E}^{vb}_B$ are disjoint for every $B$,
		\item \label{item:init} if $L\leq \mathrm{diam}(\pazocal{C}_o(h))<\infty$, then $\pazocal{E}_B:=\pazocal{E}^{b}_B\cup\pazocal{E}^{vb}_B$ happens for every $B\in \partial\pazocal{C}_o(h,L)$,
		\item \label{item:prop} if a pair $B,B'$ of neighbouring boxes lies in $\pazocal{C}_o(h,L)$ and $\pazocal{E}^{b}_B$ happens, then $\pazocal{E}_{B'}$ happens.
	\end{enumerate}
\end{definition}

For our purposes, both $\pazocal{E}^{b}_B$ and $\pazocal{E}^{vb}_B$ will be chosen to be unlikely events, with $\pazocal{E}^{vb}_B$ in particular being extremely unlikely, in the sense that its probability decays exponentially in $\cp(B)$.  Item \ref{item:init} can be thought of as an \emph{initiation property} that ensures that the union of the boxes $B\in \pazocal{C}_o(h,L)$ for which $\pazocal{E}_B$ happens, has capacity at least $\cp(\pazocal{C}_o(h))$. Item \ref{item:prop} can be thought of as a \emph{propagation property}. Ideally, we would like the event $\pazocal{E}^{vb}_B$ to happen for most boxes in $\partial\pazocal{C}_o(h,L)$. If this is not the case, then we have many boxes $B\in\partial\pazocal{C}_o(h,L)$ for which $\pazocal{E}^{b}_B$ happens. In this case, item \ref{item:prop} ensures that for many boxes in $\pazocal{C}_o(h,L)$ that are adjacent to $\partial\pazocal{C}_o(h,L)$, the event $\pazocal{E}_B$ happens. Continuing in this way we explore more and more boxes for which $\pazocal{E}_B$ happens.

With such events in hand, we will associate to $\pazocal{C}_o(h)$ an interface $\pazocal{I}$ such that for each box $B$ of $\pazocal{I}$, $\pazocal{E}_B$ happens. An \emph{interface} $\pazocal{I}$ is a finite collection of disjoint boxes of $L_1\Z^d, L_2\Z^d, \ldots L_k\Z^d$ for an integer $k>0$ and $1\leq L_1<L_2<\ldots<L_k$.
Most of the $\pazocal{I}$ we will consider, $o$ will be contained in a bounded component of $\Z^d \setminus \pazocal{I}$ (thus the term ``interface''), but it will be more convenient for us not to add this condition in the definition. When $\pazocal{E}_B$ happens for each box $B$ of $\pazocal{I}$, we will say that $\pazocal{I}$ \emph{occurs}. There are two subsets of $\pazocal{I}$ that play an important role. The first one, denoted $\pazocal{B}$, is the set of boxes $B\in I$ such that $\pazocal{E}^{b}_B$ happens. The second one, denoted $\pazocal{VB}$, is the set of all boxes $B\in \pazocal{I}$ such that $\pazocal{E}^{vb}_B$ happens. 

In the following theorem, we construct a family of interfaces $\mathcal{I}_N$ of small cardinality such that whenever $A_N(h)$ happens, some interface $\pazocal{I}\in \mathcal{I}_N$ occurs for which either $\pazocal{VB}$ has large capacity or $\pazocal{B}$ has large cardinality. 

\begin{theorem}[Multi-scale coarse graining]\label{thm:multi-int}
	Let $\pazocal{E}^{i}_B$, $i\in\{b,vb\}$, $B\in L\Z^d$, $L\geq1$, be a family of events which are admissible for each $L\geq1$. For every $\rho>0$ and $\delta>0$, there exist constants $0<t=t(d,\rho,\delta)<1$, $L_0=L_0(d,\rho,\delta)>0$, $N_0=N_0(d,\rho,\delta)>0$ such that for every $N\geq N_0$, there is a family $\mathcal{I}_N$ of interfaces such that the following hold:
	\begin{enumerate}[(a)]
		\item $|\mathcal{I}_N|\leq e^{\delta tN}$,
		\item for every $\pazocal{I}\in \mathcal{I}_N$, we have $L_0\leq L_1$ and $|\pazocal{I}|\leq \delta t N$,
		\item on the event $A_N(h)$, some $\pazocal{I}\in \mathcal{I}_N$ occurs with $L_k\leq \mathrm{diam}(\pazocal{C}_o(h))$ and $\pazocal{B}\subset L_1\Z^d$, and one of the following holds:
		\begin{enumerate}[label=(c\arabic*)]
			\item \label{item:c1} $\cp\big(\bigcup_{B\in \pazocal{VB}} B\big)\geq N/4d$,
			\item \label{item:c2} $|\pazocal{B}|L_1^{\rho}\geq tN$.
		\end{enumerate}
	\end{enumerate}
\end{theorem}

We stress that the constants $t,L_0$ and $N_0$ in the above theorem depend only on $d$, $\rho$ and $\delta$ and not on the choice of $\pazocal{E}^{b}_B$ and $\pazocal{E}^{vb}_B$. We also remark that for our applications, $\pazocal{E}_B$ will be chosen in such a way that its probability decays stretched exponentially with exponent the constant $\rho$ appearing in the statement of the theorem.

\subsubsection{Exponential decay in the supercritical regime}\label{subsec:exp_decay_sup}

We will split the proof of \Tr{thm:exp_Zd} into two parts, depending on whether $h$ belongs to the supercritical or the subcritical regime. We will first handle the supercritical regime. Our aim is to choose $\pazocal{E}^{b}_B$ and $\pazocal{E}^{vb}_B$ appropriately and then apply \Tr{thm:multi-int}. 

To this end, consider an integer $L>0$ and a box $B\in L\Z^d$. We define $L_0=L/M\approx L^{\frac{1}{d-1}}\log(L)$, where $M=\left\lfloor L^{\frac{d-2}{d-1}}/\log(L)\right\rfloor$.  A connected subgraph of $B$ is called \emph{dense} if it intersects at least $\frac{3}{4}M^d$ boxes of $L_0\Z^d$ contained in $B$ and has diameter at least $L/5$ -- the latter follows immediately for any connected subgraph that intersects at least $\frac{3}{4}M^d$ boxes contained in $B$, provided that $L$ is large enough, but we will not need this fact.

Fix $h'<h_*$ and $\varepsilon_0:=(h_*-h')/2$. For any $h\leq h'$, let $\pazocal{E}^{b}_B$ be the intersection of the events

\begin{enumerate}[label=(b\arabic*)]
	\item \label{item:bad_1} for every $B'$ which is either $B$ or a neighbour of $B$, $\{\varphi\geq h\}\cap B'$ contains a dense cluster,
	\item \label{item:bad_2} $\{\varphi\geq h\}\cap B$ contains a dense cluster that is not contained in $\pazocal{C}_o(h)$,
	\item \label{item:bad_3} $B$ is $(\xi,\varepsilon_0)$-good,
\end{enumerate}
and $\pazocal{E}^{vb}_B$ be the union of the events

\begin{enumerate}[label=(vb\arabic*)]
	\item \label{item:very_bad_1} for some $B'$ which is either $B$ or a neighbour of $B$, $\{\varphi\geq h\}\cap B'$ does not contain a dense cluster,
	\item \label{item:very_bad_2} all dense clusters of $\{\varphi\geq h\}\cap B$ are contained in $\pazocal{C}_o(h)$, but for some neighbouring box $B'$ of $B$, $\{\varphi\geq h\}\cap B'$ contains a dense cluster that is not contained in $\pazocal{C}_o(h)$,
	\item \label{item:very_bad_3} $B$ is $(\xi,\varepsilon_0)$-bad.
\end{enumerate}

We shall verify that the family of events $\pazocal{E}^{i}_B$ is $h$-admissible. It is straightforward to verify that $\pazocal{E}^{b}_B$ and $\pazocal{E}^{vb}_B$ are disjoint for every $B$, so that \ref{item:disj} holds. Let us verify \ref{item:init}. Consider a box $B\in \partial \pazocal{C}_o(h,L)$. If $\pazocal{E}^{vb}_B$ happens, then there is nothing to show. If $\pazocal{E}^{vb}_B$ does not happen, then the non occurrence of \ref{item:very_bad_1} and \ref{item:very_bad_3} directly implies the occurrence of \ref{item:bad_1} and \ref{item:bad_3}, respectively. It remains to check that \ref{item:bad_2} holds. Let $B'\in L\Z^d\setminus \pazocal{C}_o(h,L)$ be a neighbour of $B$. Since \ref{item:bad_1} happens, $\{\varphi\geq h\}\cap B'$ contains a dense cluster, which in turn is not contained in $\pazocal{C}_o(h)$ as $B'$ is disjoint from it. From this and our assumption that \ref{item:very_bad_2} does not happen, we can conclude that \ref{item:bad_2} happens, as we wanted. Finally, let us verify \ref{item:prop}. Consider two neighboring boxes $B,B''\in \pazocal{C}_o(h,L)$ such that $\pazocal{E}^{b}_B$ happens. If $\pazocal{E}^{vb}_{B''}$ happens, then there is nothing to show. Otherwise, \ref{item:bad_1} and \ref{item:bad_3} clearly happen for $B''$ in place of $B$. It is not hard to see that property \ref{item:bad_2} happens for $B''$, since \ref{item:bad_2} happens for $B$ and \ref{item:very_bad_2} does not happen for $B''$. 

The events appearing in \ref{item:bad_2}, \ref{item:very_bad_1} and \ref{item:very_bad_2} are unlikely to happen. However, it will be convenient for us to work with events that, in addition to being unlikely, are independent on different boxes that are far away from each other. For this reason, we will now introduce certain local bad and very-bad events. In what follows, given a box $B=B_L(z)$, $U$ stands for $U_L(z)$ and $D$ stands for $D_L(z)$.

We say that $B$ is $(\psi,h,\varepsilon)$-\emph{good} if for every function $g:\overline{D}\rightarrow \R$ which is harmonic in $D$ and $|g(x)|<\varepsilon$ for all $x\in D$, the following happen:
\begin{itemize}
	\item $\{\psi^B+g\geq h\}\cap U$ contains a cluster of diameter at least $L/5$,
	\item for every pair $\pazocal{C}_1,\pazocal{C}_2$ of clusters of $\{\psi^B+g\geq h\}\cap U$ of diameter at least $L/5$, there is a path in $\{\psi^B+g\geq h\}\cap D$ connecting $\pazocal{C}_1$ to $\pazocal{C}_2$.
\end{itemize}  
If $B$ is not $(\psi,h,\varepsilon)$-good, we will call it $(\psi,h,\varepsilon)$-\emph{bad}. It is not hard to see that if $\pazocal{E}^{b}_B$ happens and $L\leq \mathrm{diam}(\pazocal{C}_o(h))$, then $B$ is $(\psi,h,\varepsilon_0)$-bad (with the choice $g=\xi^B$), since $\pazocal{C}_o(h)\cap U$ contains a cluster of diameter at least $L/5$.
The following result will be proved in Section~\ref{sec:decay_bad}.

\begin{proposition}[Decay of badness]\label{prop:psi-bad}
	For every $h'<h_*$ and $0<\varepsilon< h_*-h'$, there exist constants $c_1=c_1(h',\varepsilon)>0$ and $\rho=\rho(d)>0$ such that for every $h\leq h'$ and $L\geq 1$,
	$$\Pr[\text{$B_L$ is $(\psi,h,\varepsilon)$-bad}]\leq e^{-c_1 L^{\rho}}.$$
\end{proposition} 

We now define another local event. We say that $B$ is $(\psi,h,\varepsilon)$-\emph{very-good} if for every function $g:\overline{D}\rightarrow \R$ which is harmonic in $D$ and $|g(x)|<\varepsilon$ for all $x\in D$, the following happen: 
\begin{itemize}
	\item for every $B'$ which is either $B$ or some neighbour of $B$, $\{\psi^B+g\geq h\}\cap B'$ contains a dense cluster,
	\item for every neighbour $B''$ of $B$ and every pair of dense clusters of $\{\psi^B+g\geq h\}\cap B$ and $\{\psi^B+g\geq h\}\cap B''$, respectively, there is a path in $\{\psi^B+g\geq h\}\cap D$ visiting both dense clusters.
\end{itemize} 
If $B$ is not $(\psi,h,\varepsilon)$-very-good, we will call it $(\psi,h,\varepsilon)$-\emph{very-bad}. It is not hard to see that if $\pazocal{E}^{vb}_B$ happens and $B$ is $(\xi,\varepsilon_0)$-good, then $B$ is $(\psi,h,\varepsilon_0)$-very-bad. The following result will be proved in Section~\ref{sec:decay_verybad}.

\begin{proposition}[Decay of very-badness]\label{prop:psi-very-bad}
	For every $h'<h_*$ and $0<\varepsilon<h_*-h'$, there exist a constant $c_2=c_2(h',\varepsilon)>0$ such that for every $h\leq h'$ and $L\geq 1$,
	$$\Pr[\text{$B_L$ is $(\psi,h,\varepsilon)$-very-bad}]\leq e^{-c_2 L^{d-2}}.$$
\end{proposition}

Assuming \Tr{thm:multi-int} and Propositions~\ref{prop:psi-bad} and \ref{prop:psi-very-bad}, we are now in position to prove \Tr{thm:exp_Zd} for $h$ in the supercritical regime.

\begin{proof}[Proof of Theorem~\ref{thm:exp_Zd} for $h<h_*$.]
	Consider some $h\leq h'<h_*$ and let $\rho>0$ be the exponent of Proposition \ref{prop:psi-bad}. Consider also a constant $\delta>0$ which will be chosen along the way to be sufficiently small. We start by applying \Tr{thm:multi-int} for the choice of events $\pazocal{E}^{b}_B$ and $\pazocal{E}^{vb}_B$ mentioned above to obtain a family $\mathcal{I}_N$ as in the statement of the theorem. For each $\pazocal{I}\in \mathcal{I}_N$, we will prove an exponential upper bound for the probability that $\pazocal{I}$ occurs satisfying either \ref{item:c1} or \ref{item:c2} and then apply a union bound over all $\pazocal{I}\in\mathcal{I}_N$.
		
	First, let us fix $\pazocal{I}\in\mathcal{I}_N$ and a pair of subsets $\pazocal{I}_1,\pazocal{I}_2\subset\pazocal{I}$ such that $\pazocal{I}_1$ satisfies $\cp\big(\bigcup_{B\in \pazocal{I}_1} B\big)\geq N/4d$ and $\pazocal{I}_2$ satisfies $\pazocal{I}_2\subset L_1\Z^d$ (where $L_1$ is the smallest scale of $\pazocal{I}$) and $|\pazocal{I}_2|L_1^{\rho}\geq tN$. We will bound separately the probability that $\pazocal{VB}=\pazocal{I}_1$ and $\pazocal{B}=\pazocal{I}_2$. We start by the latter. Let $\pazocal{I}'_2$ be a well-separated subset of $\pazocal{I}_2$ that is maximal with respect to this property. By the maximality of $\pazocal{I}'_2$, for every $B_{L_1}(z)\in \pazocal{I}_2$, there is some $B_{L_1}(w)\in S$ such that $K_{L_1}(z)\cap K_{L_1}(w)\neq \emptyset$, hence $|\pazocal{I}'_2|\geq 201^{-d}|\pazocal{I}_2|$. Notice that the local fields $\psi^B$, $B\in \pazocal{I}'_2$ are independent of each other (since $\pazocal{I}'_2$ is well-separated) and each box in $\pazocal{I}'_2$ is $(\psi,h,\varepsilon_0)$-bad (since the boxes of $\pazocal{I}$ have scale smaller than the diameter of $\pazocal{C}_o(h)$). Therefore, by Proposition \ref{prop:psi-bad} and independence, we have
	$$\Pr[\pazocal{I} \text{ occurs with }\pazocal{B}=\pazocal{I}_2]\leq \exp\{-201^{-d}c_1tN\}.$$
	
	We shall now bound the probability that $\pazocal{VB}=\pazocal{I}_1$. First, we restrict $\pazocal{I}_1$ to a well-separated subset with capacity of order $N$. Let $L_1<L_2<\ldots<L_k$ be the scales of $\pazocal{I}_1$. Let $\pazocal{I}_1^k$ be a subset of $\pazocal{I}_1\cap L_k\Z^d$ which is well-separated and maximal with respect to this property. Proceeding inductively, for each $i\in \{1,2,\ldots,k\}$, let $\pazocal{I}_1^i$ be a subset of $\pazocal{I}_1\cap L_i\Z^d$ such that $\bigcup_{j=i}^k \pazocal{I}_1^j$ is well-separated and $\pazocal{I}_1^i$ is a maximal set with respect to this property. Finally, let $\pazocal{I}'_1=\bigcup_{j=1}^k \pazocal{I}_1^j$. It follows from the maximality of the construction that for every $B\in \pazocal{I}_1$ of scale $L_i$ there exists $B'\in \pazocal{I}'_1$ of scale $L_j\geq L_i$ such that the $\Vert \cdot \Vert_{\infty}$-distance between $B$ and $B'$ is at most $201L_j$. In this case, for every $x\in B$ we have that $\Pr_x[H_{B'}<\infty] \geq q$, where $q=q(d)>0$ is a constant depending only on the dimension $d$ -- see e.g.~\cite[Proposition 2.2.2]{Law91}. 
	It then follows from the sweeping identity \eqref{sweeping} that $\cp(\Sigma(\pazocal{I}'_1))\geq q\cp(\Sigma(\pazocal{I}_1))\geq qN/4d$.
	
	Notice that each box in $\pazocal{I}'_1$ is either $(\xi,\varepsilon_0)$-bad or $(\psi,h,\varepsilon_0)$-very-bad. Let $\Xi(\pazocal{I}'_1)$ be the (random) union of the boxes in $\pazocal{I}'_1$ that are $(\xi,\varepsilon_0)$-bad and let $\Psi(\pazocal{I}'_1)$ be the (random) union of the boxes in $\pazocal{I}'_1$ that are $(\psi,h,\varepsilon_0)$-very-bad. By the sub-additivity of the capacity, 
	$$\text{either} \quad \cp(\Xi(\pazocal{I}'_1))\geq \frac{qN}{8d} \quad \text{or} \quad \cp(\Psi(\pazocal{I}'_1))\geq \frac{qN}{8d}.$$ 
	Applying \Lr{lem:multi_Szn} and a union bounded over all possibilities for $\Xi(\pazocal{I}'_1)$ we obtain
	\begin{align*} 
	\Pr\left[\pazocal{I} \text{ occurs with } \pazocal{VB}=\pazocal{I}_1 \text{ and }\cp(\Xi(\pazocal{I}'_1))\geq \frac{qN}{8d}\right]&\leq \sum_{\pazocal{J}} \exp\left\{-c_0\varepsilon_0^2\cp(\pazocal{J})\right\} \\
	&\leq 2^{\delta tN}\exp\left\{\frac{-c_0\varepsilon_0^2 qN}{8d}\right\},
	\end{align*}
	where the sum ranges over all possible $\pazocal{J}$ such that $\cp(\pazocal{J})\geq \frac{qN}{8d}$. Recall that $|\pazocal{J}|\leq |\pazocal{I}'_1|\leq |\pazocal{I}|\leq \delta t N$, so that we can indeed guarantee that $\pazocal{J}$ satisfies the hypothesis of Lemma~\ref{lem:multi_Szn} by decreasing the value of $\delta$ if necessary. The term $2^{\delta tN}$ above accounts for the number of possible $\pazocal{J}$. For the second case, notice that 
	$$\cp(\Psi(\pazocal{I}'_1))\leq \sum_{B_{L_i}(z)\in \Psi(\pazocal{I}'_1)}\cp(B_{L_i}(z))\leq C\sum_{B_{L_i}(z)\in \Psi(\pazocal{I}'_1)}L_i^{d-2}$$ 
	by the sub-additivity of the capacity and \eqref{ball cap}. Hence by Proposition~\ref{prop:psi-very-bad}, we have 
	\begin{align*} 
	\Pr\left[\pazocal{I} \text{ occurs with } \pazocal{VB}=\pazocal{I}_1 \text{ and } \cp(\Psi(\pazocal{I}'_1))\geq \frac{qN}{8d}\right]&\leq \sum_{\pazocal{J}}\exp\left\{-c_2\sum_{B_{L_i}(z)\in \pazocal{J}}L_i^{d-2}\right\}\\ 
	&\leq 2^{\delta tN}\exp\left\{-\frac{c_2qN}{8Cd}\right\}.
	\end{align*}
	
	Since $|\mathcal{I}_N|\leq e^{\delta tN}$, applying a union bound over all $\pazocal{I}\in \mathcal{I}_N$ and all possible $\pazocal{I}_1,\pazocal{I}_2\subset \pazocal{I}$, and decreasing $\delta$ even further, if necessary, we obtain that
	\begin{align*}
		\Pr[A_N(h)]&\leq \exp\{\delta tN\}2^{\delta tN}\left(\exp\{-201^{-d}c_1tN\}+2^{\delta tN}\exp\left\{\frac{-c\varepsilon_0^2 qN}{8d}\right\}+2^{\delta tN}\exp\left\{-\frac{c_2qN}{8Cd}\right\}\right)\\ &\leq\exp\{-c'N\}
	\end{align*}
	for some constant $c'>0$ depending only on $h'$ and $d$, as desired.
\end{proof}

\subsubsection{Exponential decay in the subcritical regime}\label{subsec:exp_decay_sub}

We now move on to the proof of \Tr{thm:exp_Zd} for $h$ in the subcritical regime. We will implement a strategy similar to the one we used for the subcritical regime. 

First, we need to choose suitably the events $\pazocal{E}^{b}_B$ and $\pazocal{E}^{vb}_B$. Given $h\geq h'>h_*$, $\varepsilon_0=(h'-h_*)/2$ and a box $B\in L\Z^d$, let $\pazocal{E}_B$ be the event that $\{\varphi\geq h\}\cap U$ contains a cluster of diameter at least $L/5$, and let 
$\pazocal{E}^{b}_B\coloneqq \pazocal{E}_B \cap \{\text{$B$ is $(\xi,\varepsilon_0)$-good}\}$ and $\pazocal{E}^{vb}_B\coloneqq \pazocal{E}_B \cap \{\text{$B$ is $(\xi,\varepsilon_0)$-bad}\}$. It is straightforward to see that this family of events is $h$-admissible when $\pazocal{C}_o(h)$ has diameter at least $L$, since then for every box $B\in \pazocal{C}_o(h,L)$, the event $\pazocal{E}_B$ happens. 

Notice that when the event $\pazocal{E}^{b}_B$ happens, $\{\psi^B\geq h-\varepsilon_0\}\cap U$ contains a cluster of diameter at least $L/5$. The latter happens with probability decaying stretched exponentially.

\begin{proposition}\label{prop:bad_sub}
	For every $h'>h_*$, there exist constants $c_3=c_3(h',d)>0$ and $\rho=\rho(d)>0$ such that for every $h\geq h'$ and $L\geq 1$,
	$$\Pr[\{\psi^B\geq h\}\cap U \text{ contains a cluster of diameter at least $L/5$}]\leq e^{-c_3 L^{\rho}}.$$
\end{proposition}
\begin{proof}
	This is a simple consequence of the (subcritical) sharpness of GFF percolation on $\Z^d$ (i.e.~$h_*=h_{**}$) mentioned in the introduction. Indeed, by the main result of \cite{DGRS20}, for every $h>h_*$, there exist $\rho=\rho(d)\in(0,1)$ and $c=c(d,h)$ such that for every $N \geq 1$, 
	\begin{equation}\label{eq:shapness_sub}
		\Pr[o\xleftrightarrow[]{\varphi\geq h}\partial B_N]\leq e^{-cN^{\rho}}.
	\end{equation} 
Assume that $\{\psi^B\geq h\}\cap D$ contains a cluster of diameter at least $L/5$ and let $\varepsilon_0=(h'-h_*)/2$. Up to a probability decaying exponentially in $L^{d-2}$, $B$ is $(\xi,\varepsilon_0)$-good by \eqref{eq:single_bad}. When this happens, $\{\varphi\geq h-\varepsilon_0\}\cap D$ contains a cluster of diameter at least $L/5$. The latter event has probability decaying stretched exponentially in the subcritical regime by \eqref{eq:shapness_sub}.
\end{proof}

Assuming \Tr{thm:multi-int} and Proposition~\ref{prop:bad_sub}, we are now in position to prove \Tr{thm:exp_Zd} for $h$ in the subcritical regime.

\begin{proof}[Proof of Theorem~\ref{thm:exp_Zd} for $h>h_*$.]
	The proof is similar to that of the case $h<h_*$ presented in Section~\ref{subsec:exp_decay_sup}. Consider some $h\geq h'> h_*$, and let $\rho>0$ be the exponent of Proposition \ref{prop:bad_sub}. Consider also a small enough constant $\delta>0$. We can apply \Tr{thm:multi-int} to obtain a family $\mathcal{I}_N$ satisfying the conclusion of the theorem. Fix $\pazocal{I}\in\mathcal{I}_N$ and a pair of subsets $\pazocal{I}_1,\pazocal{I}_2\subset\pazocal{I}$ as in the proof of the case $h<h_*$. If $\pazocal{B}=\pazocal{I}_2$ happens, then we restrict to a maximal well-separated subset $\pazocal{I}'_2$ of $\pazocal{I}_2$. Arguing as in the previous section and using Proposition~\ref{prop:bad_sub} and independence, we deduce that 
	$$\Pr[\pazocal{I} \text{ occurs with }\pazocal{B}=\pazocal{I}_2]\leq \exp\{-201^{-d}c_3tN\}.$$
	If $\pazocal{VB}=\pazocal{I}_1$ happens, then we restrict to a well-separated subset $\pazocal{I}'_1$ of $\pazocal{I}_1$ defined as in the proof of \Tr{thm:exp_Zd} for which we have $\cp(\Sigma(\pazocal{I}'_1))\geq q\cp(\Sigma(\pazocal{I}_1))\geq qN/4d$. Then we apply \Lr{lem:multi_Szn} to conclude that 
	$$\Pr[\pazocal{I} \text{ occurs with } \pazocal{VB}=\pazocal{I}_1]\leq \exp\left\{-\frac{c_0\varepsilon_0^2 qN}{4d}\right\}.$$
	
	A union bound over all $\pazocal{I}\in\mathcal{I}_N$ and over all possible subsets $\pazocal{I}_1,\pazocal{I}_2$ of $\pazocal{I}$ gives 
	$$\Pr[A_N(h)]\leq 2^{\delta tN} \exp\{\delta tN\}\left(\exp\{-201^{-d}c_3tN\}+\exp\left\{-\frac{c_0\varepsilon_0^2 qN}{4d}\right\}\right)\leq \exp\{-c' N\}$$
	for some constant $c'>0$ depending only on $h'$ and $d$, as desired.
\end{proof}

\section{Multi-scale coarse graining construction}\label{sec:proof_coarse_graining}

We will now proceed with the proof of \Tr{thm:multi-int}. In order to prove the theorem we need to introduce some notation. For every $L\geq 1$,
let $\pazocal{B}(L)$ and $\pazocal{VB}(L)$ be the set of boxes $B\in\pazocal{C}_o(h,L)$ such that $\pazocal{E}^{b}_B$ and $\pazocal{E}^{vb}_B$ happens, respectively. Finally, define $\pazocal{I}(L):=\pazocal{B}(L)\cup\pazocal{VB}(L)$. Notice that by properties (i)-(iii) of Definition~\ref{def:bad_events}, we know that if $\mathrm{diam}(\pazocal{C}_o(h))\geq L$ then $\pazocal{B}(L)\cap\pazocal{VB}(L)=\emptyset$, $\partial\pazocal{C}_o(h,L)\subset \pazocal{I}(L)$ and $\partial^{out} \pazocal{B}(L) \cap \pazocal{C}_o(h,L)\subset \pazocal{VB}(L)$. 

The following lemma will be used in the proof of \Tr{thm:multi-int}. For simplicity, we may henceforth identify any set of boxes $\pazocal{F}$ with its union $\Sigma(\pazocal{F})=\bigcup_{B\in \pazocal{F}} B$.

\begin{lemma}\label{lem:cutset}
	If the event $A_N(h)$ happens and in addition $\mathrm{diam}(\pazocal{C}_o(h))\geq L$, then we have $\cp\left(\pazocal{VB}(L)\cup \big(\partial \pazocal{C}_o(h)\cap \pazocal{B}(L)\big)\right)\geq N/2d$.
\end{lemma}
\begin{remark}
	In general, $\partial \pazocal{C}_o(h)$ is not contained entirely in $\pazocal{I}(L)$. See Figure~\ref{fig-int}.
\end{remark}
\begin{proof}
	We will show that $X\coloneqq \pazocal{VB}(L)\cup \big(\partial \pazocal{C}_o(h)\cap \pazocal{B}(L)\big)$ is a \emph{separating set} of $\pazocal{C}_o(h)$, namely that for every $x\in \pazocal{C}_o(h)$, any infinite path starting from $x$ must visit eventually $X$. 
	
	We first partition $\pazocal{C}_o(h)$ into $\pazocal{C}_o(h)\cap \pazocal{VB}(L)$, $\pazocal{C}_o(h)\cap \pazocal{B}(L)$ and $\pazocal{C}_o(h)\setminus \pazocal{I}(L)$. It is clear that $X$ is a separating set of $\pazocal{C}_o(h)\cap \pazocal{VB}(L)$. Let us show that $X$ is also a separating set of $\pazocal{C}_o(h)\cap \pazocal{B}(L)$. Indeed, each box of $\partial^{out} \pazocal{B}(L)$ lies either in $L\Z^d\setminus \pazocal{C}_o(h,L)$ or in $\pazocal{C}_o(h,L)$, and in the latter case, it must lie in $\pazocal{VB}(L)$ by property (iii) of Definition~\ref{def:bad_events}. With this observation in mind, consider an infinite path $\gamma$ starting from some vertex in $\pazocal{C}_o(h)\cap \pazocal{B}(L)$. If $\gamma$ eventually visits $\pazocal{VB}(L)$, then there is nothing to show. If $\gamma$ does not visit $\pazocal{VB}(L)$, then  consider the subpath $\gamma'$ of $\gamma$ up to the first vertex $u\in \partial \pazocal{C}_o(h)$ that $\gamma$ visits. Then $\gamma'$ visits only vertices in $\pazocal{C}_o(h,L)$ and by our assumption, it does not visit any vertices in $\partial^{out} \pazocal{B}(L)$ because $\partial^{out} \pazocal{B}(L)\cap \pazocal{C}_o(h,L)\subset \pazocal{VB}(L)$. Thus all vertices of $\gamma'$ lie in $\pazocal{B}(L)$. In particular, this holds for $u$, hence $u\in\pazocal{B}(L)\cap \partial \pazocal{C}_o(h)\subset X$.
	
	It remains to consider $\pazocal{C}_o(h)\setminus \pazocal{I}(L)$.
	First, notice that for every component $S$ of $\pazocal{C}_o(h,L)\setminus \pazocal{I}(L)$, we have $\partial^{out} S \subset \pazocal{I}(L)$ because $\partial \pazocal{C}_o(h,L)\subset \pazocal{I}(L)$. Moreover, $\partial^{out} S\cap \pazocal{B}(L)=\emptyset$ because otherwise some box of $S$ would belong to $\pazocal{I}(L)$ by property (iii) of Definition~\ref{def:bad_events}. Thus $\partial^{out} S\subset \pazocal{VB}(L)$, which implies that $\pazocal{VB}(L)$ is a separating set of $\pazocal{C}_o(h,L)\setminus \pazocal{I}(L)$, hence it is a separating set of $\pazocal{C}_o(h)\setminus \pazocal{I}(L)$, as desired. 
	
	We can now easily deduce that $\cp(X)\geq \cp(\pazocal{C}_o(h))$. Notice that $$\cp(\pazocal{C}_o(h))\geq \cp(\partial^{out} \pazocal{C}_o(h))/2d\geq N/2d$$ because when we start a simple random walk from some $x\in \partial \pazocal{C}_o(h)$, one way to never visit $\pazocal{C}_o(h)$ again is to first visit a given neighbour $y\in \partial^{out} \pazocal{C}_o(h)$ and from there to never visit $\pazocal{C}_o(h)\cup \partial^{out} \pazocal{C}_o(h)$ again.
\end{proof}

\begin{figure}
	\centering
	\includegraphics[width=0.5\linewidth]{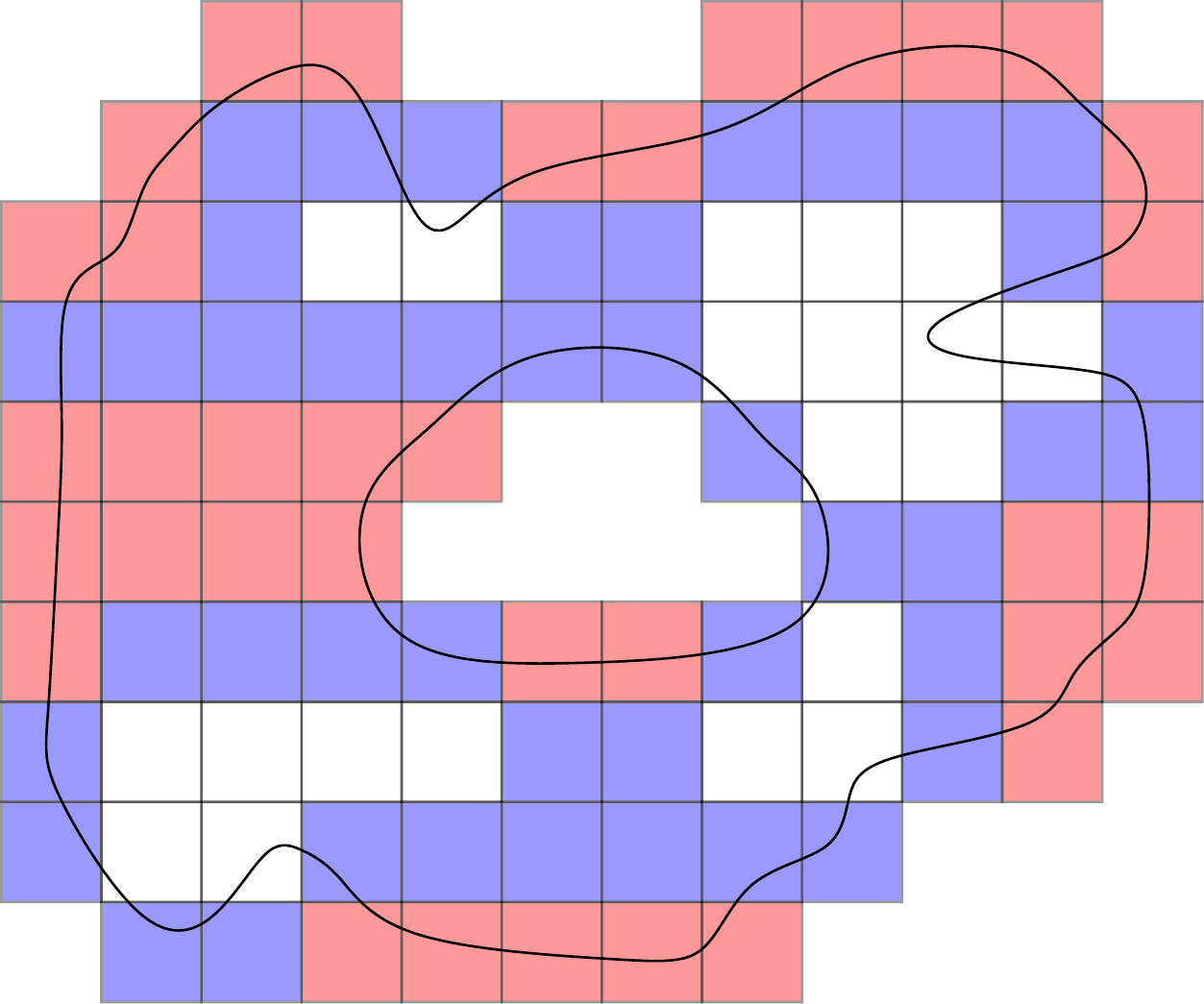}
	\caption{An illustration of an interface $\pazocal{I}(L)$. Each box of $\pazocal{C}_o(h,L)$ is depicted by a square. Red boxes belong to $\pazocal{B}(L)$, blue boxes belong to $\pazocal{VB}(L)$ and uncoloured boxes belong to $\pazocal{C}_o(h,L)\setminus \pazocal{I}(L)$. The two curves depict $\partial \pazocal{C}_o(h)$.}\label{fig-int}
\end{figure}

We are now ready to prove \Tr{thm:multi-int}.

\begin{proof}[Proof of \Tr{thm:multi-int}]
	Our aim is to construct an occurring multi-scale interface $\pazocal{I}$ for every configuration on the event $A_N(h)$. We will construct $\pazocal{I}$ by starting from $\pazocal{I}(2^k)$ for a certain choice of $2^k$ and then adding boxes of smaller and smaller scales. We will divide the definition of $\pazocal{I}$ into segments. At each step of the first segment we will add at most $\dfrac{N}{f(N)}$ boxes, at each step of the second segment we will add at most $\dfrac{N}{f(f(N))}$ boxes, and so on, where $f(N)=\log^b(N)$, $b=3(d-2)/\rho$. The process will stop once we reach a scale of size roughly $L$ or if it happens that \ref{item:c1} or \ref{item:c1} is satisfied before we reach that scale. 
	
	It suffices to prove the theorem for $\rho\leq 1$. Consider an integer $L\geq 1$ and let $N\geq N_0$, where $N_0$ is a large enough constant that will be determined along the way. Assume that the event $A_N(h)$ happens and let $$k_{1,1}\coloneqq \max\left\{0\leq k\leq \log_2(\mathrm{diam}(\pazocal{C}_o(h))): |\pazocal{I}_{1,1}(2^k)|\geq \dfrac{N}{f(N)}\right\},$$
	where $\pazocal{I}_{1,1}(2^k)\coloneqq \pazocal{I}(2^k)$. Notice that $k_{1,1}$ is well-defined, since $|\pazocal{I}_{1,1}(1)|\geq|\partial \pazocal{C}_o(h)|\geq \cp(\pazocal{C}_o(h))\geq N/2d$, provided that $N_0$ is large enough so that $f(N_0)\geq 2d$. By further increasing the value of $N_0$, we can assume that $2^{k_{1,1}}\leq r\coloneqq \frac{1}{2}\mathrm{diam}(\pazocal{C}_o(h))$ because $|\pazocal{I}_{1,1}(r)|\leq 4^d$. By definition, $$|\pazocal{I}_{1,1}(L_{1,1})|< \dfrac{N}{f(N)},$$ where $L_{1,1}\coloneqq 2^{k_{1,1}+1}\leq \mathrm{diam}(\pazocal{C}_o(h))$. On the other hand, the number of boxes of $L_{1,1}\Z^d$ that contain a box of $\pazocal{I}_{1,1}(2^{k_{1,1}})$ is at least $$\dfrac{|\pazocal{I}_{1,1}(2^{k_{1,1}})|}{2^d}\geq \dfrac{N}{2^d f(N)}.$$ Thus we can add enough boxes of $\pazocal{I}_{1,1}(2^{k_{1,1}})$ to $\pazocal{I}_{1,1}(L_{1,1})$ to obtain an interface $\pazocal{I}'_{1,1}$ such that $$\dfrac{N}{2^d f(N)}\leq |\pazocal{I}'_{1,1}|<\dfrac{N}{f(N)}.$$
	
	We then naturally define $\pazocal{VB}'_{1,1}:=\big(\pazocal{VB}(L_{1,1})\cup \pazocal{VB}(2^{k_{1,1}})\big)\cap \pazocal{I}'_{1,1}$ and $\pazocal{B}'_{1,1}:=\big(\pazocal{B}(L_{1,1})\cup \pazocal{B}(2^{k_{1,1}})\big)\cap \pazocal{I}'_{1,1}$. For the set $\pazocal{VB}'_{1,1}$ we have
	\begin{equation*}
		\text{either} \quad \cp(\pazocal{VB}'_{1,1})\geq N/4d \quad \text{or} \quad  \cp(\pazocal{VB}'_{1,1})< N/4d.
	\end{equation*} 
	In the first case, the process stops because \ref{item:c1} is satisfied, and we let $\pazocal{I}=\pazocal{VB}'_{1,1}$. In the second case, we would like to check whether \ref{item:c2} is satisfied. 
	For that purpose, we consider two cases according to whether 
	\begin{equation*}
		|\pazocal{B}'_{1,1}|\geq \dfrac{N}{2^df(N)}
		\quad \text{or} \quad |\pazocal{B}'_{1,1}|< \dfrac{N}{2^df(N)}.
	\end{equation*}
	In the first case, we stop the first segment of our process. In the second case, we move on to the second step of the first segment. We remark that along the way of the second and every subsequent step, we will define some integers $k_{i,j}, L_{i,j}$ and some collections $\pazocal{I}'_{i,j}$ of $2^{k_{i,j}}$-boxes and $L_{i,j}$-boxes, where $L_{i,j}=2^{k_{i,j}+1}$. To avoid repetition, let us mention that we will use the notation $\pazocal{VB}'_{i,j}:=(\pazocal{VB}(L_{i,j})\cup \pazocal{VB}(2^{k_{i,j}}))\cap \pazocal{I}'_{i,j}$ and $\pazocal{B}'_{i,j}:=(\pazocal{B}(L_{i,j})\cup \pazocal{B}(2^{k_{i,j}}))\cap \pazocal{I}'_{i,j}$.

	For the second step, we will require $N$ to be large enough so that $f(N)\geq 4d$. Now let $$k_{1,2}=\max\{k\geq 0 : |\pazocal{I}_{1,2}(2^k)|\geq \dfrac{N}{f(N)}\},$$
	where $\pazocal{I}_{1,2}(2^k)$ is the set of boxes of $\pazocal{I}(2^k)$ that lie in some box of $\pazocal{B}'_{1,1}$. To see that $k_{1,2}$ is well-defined, notice first that 
	\begin{equation}\label{cap lower}
		\cp\big(\partial \pazocal{C}_o(h)\cap \pazocal{B}'_{1,1}\big)>N/4d
	\end{equation}
	Indeed, as $\pazocal{VB}(L_{1,1})\subset \pazocal{VB}'_{1,1}$, we obtain that $\cp\left(\pazocal{VB}'_{1,1}\cup \big(\partial \pazocal{C}_o(h)\cap \pazocal{B}'_{1,1}\big)\right)\geq N/2d$ by \Lr{lem:cutset}. The sub-additivity of capacity gives
	$$\cp\left(\pazocal{VB}'_{1,1}\cup \big(\partial \pazocal{C}_o(h)\cap \pazocal{B}'_{1,1}\big)\right)\leq \cp(\pazocal{VB}'_{1,1})+ \cp\big(\partial \pazocal{C}_o(h)\cap \pazocal{B}'_{1,1}\big).$$  Inequality \eqref{cap lower} follows now from our assumption that $\cp(\pazocal{VB}'_{1,1})< N/4d$.
	Hence $\pazocal{I}_{1,2}(1)$, which contains $\partial \pazocal{C}_o(h)\cap \pazocal{B}'_{1,1}$, has size at least $N/4d$. Moreover, by our assumption that $|\pazocal{B}'_{1,1}|< \dfrac{N}{2^df(N)}$, we obtain that $k_{1,2}<k_{1,1}$. This proves that $k_{1,2}$ is well-defined.
	
	Let now $L_{1,2}\coloneqq 2^{k_{1,2}+1}$. Arguing as in the first step, we obtain an interface 
	$\pazocal{I}'_{1,2}$ such that 
	$$\dfrac{N}{2^d f(N)}\leq |\pazocal{I}'_{1,2}|< \dfrac{N}{f(N)},$$ that is obtained from $\pazocal{I}_{1,2}(L_{1,2})$ by adding enough $2^{k_{1,2}}$-boxes of $\pazocal{I}_{1,2}(2^{k_{1,2}})$ that are disjoint from the boxes of $\pazocal{I}_{1,2}(L_{1,2})$. At this point, we take cases according to whether
	\begin{equation*}
		\cp\Big(\bigcup_{j=1}^2 \pazocal{VB}'_{1,i}\Big)\geq N/4d \quad \text{or} \quad  \cp\Big(\bigcup_{j=1}^2 \pazocal{VB}'_{1,i}\Big)< N/4d.
	\end{equation*} 
	As before, if the first case happens, the process stops and we define $\pazocal{I}=\bigcup_{j=1}^2 \pazocal{VB}'_{1,i}$, while if the second case happens, then we check whether 
	\begin{equation*}
		|\pazocal{B}'_{1,2}|\geq \dfrac{N}{2^df(N)} 
		\quad \text{or} \quad |\pazocal{B}'_{1,2}|< \dfrac{N}{2^df(N)}.
	\end{equation*} 
	Similarly, if the first case happens, we end the first segment. If the second case happens, then we continue to the third step. At this point, we need a generalisation of \Lr{lem:cutset} which will ensure that $|\pazocal{I}_{1,3}(1)|\geq N/4d$ and more generally that $|\pazocal{I}_{i,j}(1)|\geq N/4d$ for the subsequent steps. This is proved in \Lr{lem:multi-cutset}.
	
	Continuing in this manner, we obtain a sequence of interfaces $(\pazocal{I}'_{1,j})_{j\geq1}$, where $\pazocal{I}'_{1,j}$ is contained in $\pazocal{B}'_{1,j-1}$. We claim that eventually for some integer $j_1\geq 1$, 
	\begin{equation*}
		\cp\Big(\bigcup_{i=1}^{j_1} \pazocal{VB}'_{1,i}\Big)\geq N/4d
		\quad \text{or} \quad 
		\dfrac{N}{2^df(N)}\leq |\pazocal{B}'_{1,j_1}|<\dfrac{N}{f(N)}.
	\end{equation*}
	Indeed, if the first inequality does not hold, then by \Lr{lem:multi-cutset} and the sub-additivity of the capacity we have $\cp(\pazocal{I}'_{1,j})>N/4d$. On the other hand, by \eqref{ball cap} and the sub-additivity of capacity again, we have $$\cp(\pazocal{I}'_{1,j})\leq C|\pazocal{I}'_{1,j}|L^{d-2}_{1,j}\leq \dfrac{CNL^{d-2}_{1,j}}{f(N)}.$$
	We thus conclude that 
	\begin{equation}\label{eq:lower bound}
		L_{1,j}\geq \Big(\dfrac{f(N)}{4dC}\Big)^{\frac{1}{d-2}}.
	\end{equation}
	However, it follows from the definitions that $(L_{1,j})_{j\geq 1}$ is a strictly decreasing sequence, and so \eqref{eq:lower bound} cannot hold for arbitrary large $j$. 
	
	We end the first segment as soon as we reach a step $j_1$ as above. We shall now decide whether we start the second segment or not. If it happens that 
	\begin{equation}\label{eq:or}
		\cp\Big(\bigcup_{j=1}^{j_1} \pazocal{VB}'_{1,j}\Big)\geq n/4d \quad \text{or} \quad  \dfrac{N}{2^df(N)}\leq |\pazocal{B}'_{1,j_1}|<\dfrac{N}{f(N)} \; \text{and} \; L_{1,j_1}\geq f(N)^{\frac{1}{\rho}},
	\end{equation}
	then our process stops. In the first case, we simply set $\pazocal{I}=\bigcup_{j=1}^{j_1} \pazocal{VB}'_{1,j}$.
	In the second case though, we set $\pazocal{I}=\pazocal{B}(L_{1,j_1})\cap \pazocal{I}'_{1,j_1}$ if $|\pazocal{B}(L_{1,j_1})\cap \pazocal{I}'_{1,j_1}|\geq |\pazocal{B}(2^{k_{1,j_1}})\cap \pazocal{I}'_{1,j_1}|$ and $\pazocal{I}=\pazocal{B}(2^{k_{1,j_1}})\cap \pazocal{I}'_{1,j_1}$ otherwise. In other words, $\pazocal{I}$ contains only one of the sets $\pazocal{B}(L_{1,j_1})\cap \pazocal{I}'_{1,j_1}$ and $\pazocal{B}(2^{k_{1,j_1}})\cap \pazocal{I}'_{1,j_1}$, namely that of larger size.
	If \eqref{eq:or} is not satisfied, then we move on to the second segment.  
	
	Arguing in a similar manner, we obtain a sequence of occurring interfaces $(\pazocal{I}'_{2,j})_{j\geq1}$
	such that $|\pazocal{I}'_{2,j}|<N/f(f(N))$ for all $j\geq 1$, where each $\pazocal{I}'_{2,j}$ lies in $\pazocal{I}'_{1,j_1}$. The segment ends when we reach a certain step $j_2$ such that either 
	\begin{equation*} 
		\cp\Big(\bigcup_{i=1}^2 \bigcup_{j=1}^{j_i} \pazocal{VB}'_{i,j}\Big)\geq N/4d \quad \text{or} \quad 
		\dfrac{N}{2^d f(f(N))}\leq |\pazocal{B}'_{2,j_2}|<\dfrac{N}{f(f(N))}.
	\end{equation*}
	The process stops at the end of the second segment if 
	\begin{equation*}
		\cp\Big(\bigcup_{i=1}^2 \bigcup_{j=1}^{j_i} \pazocal{VB}'_{i,j}\Big)\geq N/4d \quad \text{or} \quad \dfrac{N}{2^df(f(N))}\leq |\pazocal{B}'_{2,j_2}|<\dfrac{N}{f(f(N))} \; \text{and} \; L_{2,j_2}\geq f(f(N))^{\frac{1}{\rho}}.
	\end{equation*}
	In that case, we set $\pazocal{I}=\bigcup_{i=1}^2\bigcup_{j=1}^{j_i} \pazocal{VB}'_{i,j}$, $\pazocal{I}=\pazocal{B}(L_{2,j_2})\cap \pazocal{I}'_{2,j_2}$ or $\pazocal{I}=\pazocal{B}(2^{k_{2,j_2}})\cap \pazocal{I}'_{2,j_2}$, as appropriate.
	
	Proceeding inductively, we define sequences of occurring interfaces 
	$(\pazocal{I}'_{1,j})_{j=1}^{j_1},(\pazocal{I}'_{2,j})_{j=1}^{j_2},\ldots$ such that $|\pazocal{I}'_{i,j}|<N/f^{\circ i}(N)$ for all $i$ and $j$, where $f^{\circ i}$ denotes the $i$-fold composition of $f$. At the end of an arbitrary $k$th segment, we either have $\cp\Big(\bigcup_{i=1}^k \bigcup_{j=1}^{j_i} \pazocal{VB}'_{i,j}\Big)\geq N/4d$ or
	$\dfrac{N}{2^df^{\circ k}(N)}\leq |\pazocal{B}'_{i,j_i}|<\dfrac{N}{f^{\circ k}(N)}$. Let $m=m(N,L)$ be the largest integer such that $f^{\circ m}(N)> M\coloneqq d2^d CL^{d-2}$. Notice that $m$ is well-defined for every $N$ such that $f(N)> M$. If the desired conditions are not satisfied at the end of the $i$th segment for every $i\leq m$, we move on to the $(m+1)$th segment. This segment plays a special role, as we are defining each $\pazocal{I}'_{m+1,j}$ in such a way that $$\dfrac{N}{2^d M}\leq |\pazocal{I}'_{m+1,j}|< \dfrac{N}{M}.$$ At the end of the $(m+1)$th segment we have $\cp\Big(\bigcup_{i=1}^{m+1} \bigcup_{j=1}^{j_i} \pazocal{VB}'_{i,j}\Big)\geq N/4d$ or
	$\dfrac{N}{2^dM}\leq |\pazocal{B}'_{m+1,j_{m+1}}|<\dfrac{N}{M}$. 
	Finally, we set $\pazocal{I}=\bigcup_{i=1}^{m+1}\bigcup_{j=1}^{j_i} \pazocal{VB}'_{i,j}$, $\pazocal{I}=\pazocal{B}(L_{m+1,j_{m+1}})\cap \pazocal{I}'_{m+1,j_{m+1}}$ or $\pazocal{I}=\pazocal{B}(2^{k_{m+1,j_{m+1}}})\cap \pazocal{I}'_{m+1,j_{m+1}}$, as appropriate.
	
	It is not hard to see that if \ref{item:c1} is not satisfied, then \ref{item:c2} is satisfied for 
	\begin{equation}\label{eq:t}
		t=\frac{L^{\rho}}{2^{d+1} M}.
	\end{equation}
	Indeed, if the process stops at the end of the $i$th segment for some $i\leq m$, then the smallest scale of $\pazocal{I}$ is at least $\frac{1}{2}L_{i,j_i}\geq \frac{1}{2} f^{\circ i}(N)^{\frac{1}{\rho}}$ and $|\pazocal{B}|\geq \dfrac{|\pazocal{B}'_{i,j_i}|}{2}\geq \dfrac{N}{2^{d+1} f^{\circ i}(N)}$ (here we use the notation introduced above the statement of \Tr{thm:multi-int}). Thus $|\pazocal{B}|L^{\rho}_k \geq 2^{-\rho-d-1}N$. On the other hand, if the process stops at the end of the $(m+1)$th segment, then we can argue as in the proof of \eqref{eq:lower bound} to deduce that $$L_{m+1,j_{m+1}}\geq \Big(\dfrac{M}{4dC}\Big)^{\frac{1}{d-2}}=2L.$$ Thus the smallest scale of $\pazocal{I}$ is at least $\frac{1}{2}L_{m+1,j_{m+1}}\geq L$, which implies that $|\pazocal{B}|L^{\rho}_k \geq t N$. Since $t\leq 2^{-\rho-d-1}$, the desired assertion follows.
	
	The above construction gives us a family of interfaces $\mathcal{I}_N$ satisfying all the properties claimed in Theorem~\ref{thm:multi-int}. The only properties that do not follow immediately from the construction are that $|\mathcal{I}_N|\leq e^{\delta tN}$ and that  $|\pazocal{I}|\leq \delta t N$ for every $\pazocal{I}\in \mathcal{I}_N$. In order to prove these inequalities, we will treat each segment separately. We start with the first segment. To determine $\pazocal{I}'_{1,j}, j=1,2,\ldots,j_1$, we need to first determine the sequence $(L_{1,j})_{j=1}^{j_1}$. Recall that by construction we have $L_{1,1}\leq \mathrm{diam}(\pazocal{C}_o(h))$. As we mentioned above Corollary~\ref{cor:vol_decay}, a cluster of capacity at most $N$ has volume (and therefore diameter) at most $C_1N^{\frac{d}{d-2}}\leq C_1N^3$, thus $L_{1,1}\leq \mathrm{diam}(\pazocal{C}_o(h))\leq C_1N^3$. Therefore, $(L_{1,j})_{j=1}^{j_1}$ is simply a strictly decreasing sequence of powers of $2$ with exponents at most $\log_2(C_1N^3)$, which in turn implies that there are at most $2^{\log_2(C_1N^3)}=C_1N^3$ possibilities for $(L_{1,j})_{j=1}^{j_1}$. 
	Once the scales $(L_{1,j})_{j=1}^{j_1}$ are fixed, we should bound the possibilities for $\pazocal{I}'_{1,j}, 1\leq j\leq j_1$. Notice that for all $j=1,2,\ldots,j_1$, each box of $\pazocal{I}'_{1,j}$ is at distance at most $C_1N^3$ from the origin and furthermore $|\pazocal{I}'_{1,j}|\leq N_1:=\lfloor N/f(N) \rfloor$. Hence, for each $1\leq j\leq j_1$, the number possibilities for $\pazocal{I}'_{1,j}$ given $L_{i,j}$ is at most
	$$\sum_{k=1}^{N_1}{N'_1\choose{k}},$$
	where $N'_1=C_2N^{3d}$. Using the inequality ${{n}\choose{k}}\leq \left(\tfrac{n}{k}\right)^k e^k$ and the monotonicity of the combinatorial coefficient ${{n}\choose{k}}$ for $k\leq n/2$ we obtain that $$\sum_{k=1}^{N_1}{N'_1\choose{k}}\leq N_1 \left(\dfrac{N'_1}{N_1}\right)^{N_1} e^{N_1}\leq \exp\left\{3N_1\log\left(\frac{N'_1}{N_1}\right)\right\}.$$ Overall, there at most 
	\begin{equation}\label{eq:bound_seg1}
	C_1N^3\left(\sum_{k=1}^{N_1}{N'_1\choose{k}}\right)^{\log_2(C_1N^3)}\leq \exp\left\{C_3\dfrac{N\log^2(N)}{f(N)}\right\}
	\end{equation}
	possibilities for the first segment. By increasing $C_3$ of necessary, the term inside the exponential in \eqref{eq:bound_seg1} is also an upper bound for the number of boxes of the first segment contained in $\pazocal{I}$.
	
	Moving on to the second segment, first notice that all scales $(L_{2,j})_{j=1}^{j_2}$ are powers of $2$ smaller than $f(N)^{\frac{1}{\rho}}$ (recall that \eqref{eq:or} does not hold). Therefore, there are at most $f(N)^{\frac{1}{\rho}}$ possibilities for $(L_{2,j})_{j=1}^{j_2}$.
	Since $L_{1,j_1}< f(N)^{\frac{1}{\rho}}$ and every box of the second segment is contained in $\pazocal{I}'_{1,j_1}$, which in turn contains at most $N_1$ boxes, we deduce that for every $j=1,2,\ldots,j_2$, $\pazocal{I}'_{2,j}$ contains at most $N_1 f(N)^{\frac{d}{\rho}}\leq C_4 N\lfloor f(N) \rfloor ^{\frac{d}{\rho}-1}=:N'_2$ boxes. Hence for each $1\leq j\leq j_2$ the number of possibilities for $\pazocal{I}'_{2,j}$ given $L_{2,j}$ is at most 
	$$\sum_{k=0}^{N_2}{{N'_2}\choose{k}}\leq \exp\left\{3N_2\log\left(\frac{N'_2}{N_2}\right)\right\},$$ 
	where $N_2=\lfloor N/f(f(N)) \rfloor$. Overall, there are at most 
	$$f(N)^{\frac{1}{\rho}}\left(\sum_{k=0}^{N_2}{N'_2\choose{k}}\right)^{\log_2\big(f(N)^{\frac{1}{\rho}}\big)}\leq  \exp\left\{C_3\dfrac{N\log^2(f(N))}{f(f(N))}\right\}$$ possibilities for the second segment, where for the last inequality, we increase the value of $C_3$ if necessary. 
	
	Setting $g_0(N):=N$, $g_i(N):=f^{\circ i}(N)$ for $1\leq i\leq m$ and $g_{m+1}(N):=M$ (recall that $M=d2^d CL^{d-2}$), we see that for the boxes of an arbitrary $i$th segment, there are at most
	$$\exp\left\{C_3\dfrac{N\log^2(g_{i-1}(N))}{g_i(N)}\right\}$$ possibilities. 
	Overall, we deduce that
	\begin{equation}\label{eq:bound_seg_all}
		|\mathcal{I}_N|\leq\exp\left \{C_3 N \sum_{i=1}^{m+1}\dfrac{\log^2(g_{i-1}(N))}{g_i(N)} \right\}.
	\end{equation}
	Furthermore, the term inside the exponential in \eqref{eq:bound_seg_all} is an upper bound for $|\pazocal{I}|$.
	
	Therefore, it remains to prove that 
	\begin{equation}\label{eq:delta_t}
		C_3\sum_{i=1}^{m+1}\dfrac{\log^2(g_{i-1}(N))}{g_i(N)}\leq \delta t,
	\end{equation}
	provided that $L$ and $N$ are large enough (recall from \eqref{eq:t} that $t=\tfrac{L^{\rho}}{2^{d+1} M}$). 
	We start by bounding the $(m+1)$th term. By the definition of $m$, we have $f^{\circ(m+1)}\leq M$, which implies $g_m(N)=f^{\circ m}(N)\leq e^{M^{\frac{1}{b}}}$. Therefore,
	\begin{equation}\label{eq:m+1}
		\dfrac{\log^2(g_{m}(N))}{g_{m+1}(N)}\leq M^{-1+\frac{2}{b}}
	\end{equation}
	Now, let us handle the sum up to the $m$th term. First notice that for all $i\leq m$, $\tfrac{\log^2(g_{i-1}(N))}{g_i(N)}=\log^{2-b}(g_{i-1}(N))$. 
	Now, recall that $b=3(d-2)/\rho>2$ and observe that
	$\log^{2-b}(x)\leq 2^{-1}\log^{2-b}(f(x))$ for all $x\geq C_5$. Since $g_{i-1}(N)\geq g^m(N)\geq C_5$ for all $L$ and $N$ that are large enough, one readily deduces 
	\begin{equation*}
		\log^{2-b}(g_{i-1}(N))\leq 2^{-1} \log^{2-b}(g_i(N)).
	\end{equation*}
	Iterating the last inequality, we obtain that
	$\log^{2-b}(g_{i-1}(N))\leq 2^{i-m} \log^{2-b}\left(g^{m-1}(N)\right)$, which in turn implies
	\begin{equation}\label{eq:1_to_m}
		\sum_{i=1}^m \dfrac{\log^2(g_{i-1}(N))}{g_i(N)}\leq 2\log^{2-b}\left(g^{m-1}(N)\right).
	\end{equation}
	By the definition of $m$ we know that $f^{\circ m}(N)\geq M$, which implies $g_{m-1}=f^{\circ (m-1)}(N)\geq e^{M^{\frac{1}{b}}}$. Plugging this in \eqref{eq:1_to_m} gives
	\begin{equation}\label{eq:1_to_m*}
		\sum_{i=1}^m \dfrac{\log^2(g_{i-1}(N))}{g_i(N)}\leq 2M^{-1+\frac{2}{b}}.
	\end{equation}
	Combining \eqref{eq:m+1} and \eqref{eq:1_to_m*}, we deduce that
	\begin{equation*}\label{eq:1_to_m+1}
		C_3\sum_{i=1}^{m+1}\dfrac{\log^2(g_{i-1}(N))}{g_i(N)}\leq 3C_3 M^{-1+\frac{2}{b}}.
	\end{equation*}
	Recalling the definitions of $M$ and $t$, we see that $t=C_6 M^{-1+\frac{\rho}{d-2}}$. Since by definition $b=3(d-2)/\rho$, the desired inequality \eqref{eq:delta_t} follows readily as long as $\delta\geq \tfrac{3C_3}{C_6}M^{-\frac{\rho}{3(d-2)}}$, which can be guaranteed by making $L$ sufficiently large. This completes the proof.
\end{proof}

For $1\leq i\leq m+1$ and $1\leq j\leq j_i$, let $$\overline{\pazocal{VB}}_{i,j}=\left(\bigcup_{k=1}^{i-1} \bigcup_{l=1}^{j_k}{\pazocal{VB}'}_{k,l}\right) \cup \left(\bigcup_{l=1}^{j}{\pazocal{VB}'}_{i,l}\right).$$ 
We now prove the lemma mentioned in the proof of the above theorem. We recall that for convenience we identify sets of boxes with the corresponding subsets of $\Z^d$.

\begin{lemma}\label{lem:multi-cutset}
	For every $i,j\geq 1$ we have $\cp\left(\overline{\pazocal{VB}}_{i,j}\cup \big(\partial \pazocal{C}_o(h)\cap \pazocal{B}_{i,j}\big)\right)\geq N/2d$.
\end{lemma}
\begin{proof}
	As in the proof of \Lr{lem:cutset}, the desired result will follow once we show that $X_{i,j}=\overline{\pazocal{VB}}_{i,j}\cup \big(\partial \pazocal{C}_o(h)\cap \pazocal{B}_{i,j}\big)$ is a separating set of $\pazocal{C}_o(h)$. Recall the definitions of $\pazocal{I}'_{i,j}$ and $\pazocal{I}_{i,j}(L_{i,j})$. We will prove that $X_{i,j}$ is a separating set of $\pazocal{C}_o(h)$ in the special case where $\pazocal{I}'_{1,1}=\pazocal{I}_{1,1}(L_{1,1}),\pazocal{I}'_{1,2}=\pazocal{I}_{1,2}(L_{1,2}),\ldots,\pazocal{I}'_{i,j}= \pazocal{I}_{i,j}(L_{i,j})$. The general case follows easily by removing $\left(\pazocal{I}'_{1,1}\setminus \pazocal{I}_{1,1}(L_{1,1})\right)\cup \left(\pazocal{I}'_{1,2}\setminus \pazocal{I}_{1,2}(L_{1,2})\right)\cup \ldots \cup \left(\pazocal{I}'_{i,j}\setminus \pazocal{I}_{i,j}(L_{i,j})\right)$ from $X_{i,j}$. 
	
	It is clear that $X_{i,j}$ is a separating set of $\pazocal{C}_o(h)\cap \overline{\pazocal{VB}}_{i,j}$. We claim that 
	\labtequ{claim}{every box in $\partial^{out} \pazocal{B}_{i,j}$ lies either in $\overline{\pazocal{VB}}_{i,j}$ or in $L_{i,j}\Z^d\setminus \pazocal{C}_o(h,L_{i,j})$,} which implies that $X_{i,j}$ is a separating set of $\pazocal{C}_o(h)\cap \pazocal{B}_{i,j}$ by arguing as in the proof of \Lr{lem:cutset}. Indeed, for $(i,j)=(1,1)$, the claim follows from property \ref{item:prop}. 
	Proceeding inductively, assume that the statement holds for an arbitrary $(i,j)$. Let $(k,l)$ be the next pair of indices, i.e.~$(k,l)=(i,j+1)$ if $j<j_i$ or $(k,l)=(i+1,1)$ if $j=j_i$. Clearly, every box in $\partial^{out} \pazocal{B}_{k,l}$ lies either in $L_{k,l}\Z^d\setminus \pazocal{C}_o(h,L_{k,l})$, in which case there is nothing to show, or in $\pazocal{C}_o(h,L_{k,l})$. So let us consider a box $B\in \partial^{out} \pazocal{B}_{k,l}\cap \pazocal{C}_o(h,L_{k,l})$. Then $B$ has a neighbour $B'\in  \pazocal{B}_{k,l}\subset \pazocal{B}_{i,j}$, which implies that $B$ is contained entirely in $\pazocal{B}_{i,j}$ or in $\partial^{out} \pazocal{B}_{i,j}$. 
	
	If $B$ is contained in $\pazocal{B}_{i,j}$, then $B\in \pazocal{VB}_{k,l}\subset \overline{\pazocal{VB}}_{k,l}$ because of our assumption that $B\in \partial^{out} \pazocal{B}_{k,l}\cap \pazocal{C}_o(h,L_{k,l})$. Let us now assume that $B$ is contained in some box $B''\in \partial^{out} \pazocal{B}_{i,j}$. It follows from our inductive hypothesis that $B''$ lies either in $\overline{\pazocal{VB}}_{i,j}$ or in $L_{i,j}\Z^d\setminus \pazocal{C}_o(h,L_{i,j})$. Notice that $L_{i,j}\Z^d\setminus \pazocal{C}_o(h,L_{i,j})\subset L_{k,l}\Z^d\setminus \pazocal{C}_o(h,L_{k,l})$, since any box $B_1\in L_{k,l}\Z^d$ is contained in a box $B_2\in L_{i,j}\Z^d$ and if $B_1$ intersects $\pazocal{C}_o(h)$, then so does $B_2$. Thus $B$ lies either in $\overline{\pazocal{VB}}_{k,l}$ or in $L_{k,l}\Z^d\setminus \pazocal{C}_o(h,L_{k,l})$. This proves the inductive statement and the claim follows.
	
	It remains to handle $\pazocal{C}_o(h)\setminus Y_{i,j}$, where $Y_{i,j}=\overline{\pazocal{VB}}_{i,j}\cup \pazocal{B}_{i,j}$. Let $Z_{i,j}=\pazocal{C}_o(h,L_{i,j})\setminus Y_{i,j}$. Then we claim that $\partial^{out} Z_{i,j}$ is contained in $\overline{\pazocal{VB}}_{i,j}$ which implies that $\overline{\pazocal{VB}}_{i,j}$ is a separating set of $\pazocal{C}_o(h)\setminus Y_{i,j}$. We will prove the claim inductively. For $(i,j)=(1,1)$, this follows from the proof of \Lr{lem:cutset} where it is shown that for every component $S$ of $\pazocal{C}_o(h,L_{1,1})\setminus \pazocal{I}(L_{1,1})$, we have $\partial^{out} S\subset \pazocal{VB}(L_{1,1})$. 
	
	Assume that the statement holds for some $(i,j)$. We will prove it for the next pair of indices $(k,l)$. Let $S$ be a component of $Z_{k,l}$. Although $S\subset Z_{k,l}$, it is possible that some box of $S$ is contained in $\pazocal{B}_{i,j}$. Let us assume that this is the case. Then by the connectivity of $S$ and \eqref{claim}, all boxes of $S$ are contained in $\pazocal{B}_{i,j}$. Notice that $\partial^{out} S$ lies in $\pazocal{C}_o(h,L_{k,l})$ because otherwise some box $B$ of $S$ lies in $\partial^{out} \pazocal{C}_o(h,L_{k,l})\cap \pazocal{B}_{i,j}$, hence $B$ is contained in $Y_{k,l}$, which contradicts the definition of $Z_{k,l}$. From this we deduce that $\partial^{out} S\subset Y_{k,l}$. Moreover, no box of $\partial^{out} S$ lies in $\pazocal{B}_{k,l}$ because otherwise some box of $S$ lies in $\pazocal{I}_{k,l}\subset Y_{k,l}$ by our assumption that $S$ is contained in $\pazocal{B}_{i,j}$. Therefore, $\partial^{out} S$ lies in $\overline{\pazocal{VB}}_{k,l}$.
	
	Let us now assume that no boxes of $S$ are contained in $\pazocal{B}_{i,j}$. Then $S$ lies entirely in $\pazocal{C}_o(h,L_{i,j})\setminus Y_{i,j}$, since $Y_{k,l}$ contains $\overline{\pazocal{VB}}_{i,j}$. We can now apply the induction hypothesis to deduce that $\partial^{out} S$ is contained in $\overline{\pazocal{VB}}_{i,j}\subset \overline{\pazocal{VB}}_{k,l}$. This completes the inductive proof.
\end{proof}

\section{Decay of badness}\label{sec:decay_bad}

In this section, we will prove Proposition~\ref{prop:psi-bad}. We will make use of the (supercritical) sharpness of phase transition for GFF percolation \cite{DGRS20} (i.e.~$\overline{h}=h_*$). We say that a box $B=B_L(z)$,  $z\in L\Z^d$, is $(\varphi,h)$-good if there exists a connected component in $\{\varphi\geq h\}\cap B$ with diameter at least  $L/5$ and furthermore any two clusters in $\{\varphi\geq h\}\cap U$ having diameter at least $L/10$ are connected to each other in $\{\varphi\geq h\}\cap D$. By the main result of \cite{DGRS20}, for every $h'<h_*$ there exist $\rho=\rho(d)\in(0,1)$ and $c=c(d,h')$ such that for every $h\leq h'$ and $L \geq 1$, 
\begin{equation}\label{eq:sharpness_sup}
	\Pr[B_L \text{ is $(\varphi,h)$-bad}]\geq 1-e^{-cL^{\rho}}.
\end{equation}
Our aim is to express the event that a box is $(\psi,h,\varepsilon)$-bad in terms of events depending on $\varphi$, so that we can use \eqref{eq:sharpness_sup}. For this purpose, we will make use of the following classical fact about discrete harmonic functions. For any function $f:\overline{D_L(z)}\rightarrow \R$ which is harmonic in $D_L(z)$, we have that 
\begin{equation}\label{eq:gradient}
	|f(x)-f(y)|\leq C'\left \Vert f \right \Vert_{\infty}/L
\end{equation} for neighbouring $x$ and $y$ in $B_L(z)$, where $C'=C'(d)>0$ is a universal constant -- see \cite[Theorem 1.7.1]{Law91}. We shall apply this result for $f=\xi^B$ and $B$ being a $(\xi,\varepsilon)$-good box for a certain value of $\varepsilon>0$. We first need to introduce some definitions.

Consider an integer $N\geq 1$ and let $L=\left\lfloor N/M \right\rfloor \approx N^{\frac{1}{\alpha+2}}$, where $\alpha=(2d+1)^2$ and $M=N^{\frac{\alpha+1}{\alpha+2}}$. We say that a connected subgraph $\pazocal{C}$ of $U_N$ is \emph{very dense} if it has diameter at least $N/5$ and for every box $B=B_L(z)$, $z\in L\Z^d$ contained in $B_N$, $\pazocal{C}\cap B$ contains a connected subgraph of diameter at least $L/5$. 

Given $0<\varepsilon<h_*-h$, we say that a \emph{strong local uniqueness} happens in $B_N$ if $\{\varphi\geq h+\varepsilon\}\cap U_N$ contains a very dense cluster and furthermore, for every $k\in\{-\left\lceil  \varepsilon L^{\alpha}\right\rceil,\ldots, 
\left\lceil \varepsilon L^{\alpha} \right\rceil\}$, every box $B=B_L(z)$, $z\in L\Z^d$ contained in $B_N$ is $(\varphi,h-rL^{-\alpha})$-good, where $r=k-1-7dC'\varepsilon$ and $C'$ is the constant appearing in \eqref{eq:gradient}. We denote by $\text{NSLU}(h,\varepsilon,N)$ the event that strong local uniqueness \emph{does not} happen in $B_N$.

\begin{lemma}\label{lemma:NSLU}
	For every $h'<h_*$, there exist constants $c=c(h',d,\varepsilon)>0,\rho=\rho(d)>0$ such that for every $h\leq h'$ and $N\geq 1$,
	$$\Pr[\text{NSLU}(h,\varepsilon,N)]\leq e^{-c N^\rho}.$$
\end{lemma}

Consider now the boxes of $L^2\Z^d$ contained in $B_N$. Given such a box $B$, we define $\text{Conf}(h,\varepsilon,B)$ as the event that there are a set $S\subset D$ of cardinality $|S|\geq L$ and an integer $k\in\{-\left\lceil  \varepsilon L^{\alpha}\right\rceil,\ldots,$  
$\left\lceil \varepsilon L^{\alpha} \right\rceil\}$ such that $h-kL^{-\alpha}\leq \varphi_x < h-rL^{-\alpha}$ for every $x\in S$, where $r=k-1-7dC'\varepsilon$. In other words, when the event $\text{Conf}(h,\varepsilon,B)$ happens, $\varphi_x$ is confined for at least $L$ vertices in $D$. Finally, we define $$\text{Conf}(h,\varepsilon,N)\coloneqq \bigcup_{B} \text{Conf}(h,\varepsilon,B),$$ where the union is taken over all boxes of $L^2\Z^d$ contained in $B_N$.

\begin{lemma}\label{lemma:small}
	For every $h'<h_*$ and every $0<\varepsilon<h_*-h'$, there exist constants $c=c(h',\varepsilon,d)>0,\rho=\rho(d)>0$ such that for every $h\leq h'$ and $N\geq 1$,
	$$\Pr[\text{Conf}(h,\varepsilon,N)]\leq e^{-c N^{\rho}}.$$
\end{lemma}

Proposition~\ref{prop:psi-bad} follows readily by applying the following (deterministic) lemma for $\delta=(h_*-h'-\varepsilon)/2$ together with \eqref{eq:single_bad} and Lemmas~\ref{lemma:NSLU} and \ref{lemma:small} above. 

\begin{lemma}\label{lemma:psi to phi}
	Let $h<h_*$, $0<\varepsilon<h_*-h$ and $\delta<h_*-h-\varepsilon$. If the box $B_N$ is $(\psi,h,\varepsilon)$-bad, then one of the events $\{\text{$B_N$ is $(\xi,\delta)$-bad}\}$, $\text{NSLU}(h,\varepsilon+\delta,N)$ or $\text{Conf}(h,\varepsilon+\delta,N)$ happens.
\end{lemma}

We now turn to the proof of each of the above lemmas. 

\begin{proof}[Proof of \Lr{lemma:psi to phi}]
	If $B_N$ is $(\xi,\delta)$-bad or the event $\text{Conf}(h,\varepsilon+\delta,N)$ happens, then there is nothing to prove, so let us assume that $B_N$ is $(\xi,\delta)$-good and $\text{Conf}(h,\varepsilon+\delta,N)$ does not happen. We need to show that $\text{NSLU}(h,\varepsilon+\delta,N)$ happens. To this end, if $\{\varphi\geq h+\varepsilon+\delta\}\cap U_N$ does not contain a very dense cluster, then $\text{NSLU}(h,\varepsilon+\delta,N)$ happens, so let us assume that $\{\varphi\geq h+\varepsilon+\delta\}\cap U_N$ does contain a very dense cluster $\pazocal{C}_1$. 
	
	We claim that for some function $f:\overline{D_N}\rightarrow \R$ which is harmonic in $D_N$ and satisfies $|f(x)|<\varepsilon+\delta$ for every $x\in D_N$,
	\labtequ{eq:connection}{$\{\varphi +f\geq h\}\cap U_N$ contains a cluster $\pazocal{C}_2$ of diameter at least $N/5$ which is not connected to $\pazocal{C}_1$ in $\{\varphi+f\geq h\}\cap D_N$.}
	Indeed, $B_N$ is $(\psi,h,\varepsilon)$-bad, so it follows from the decomposition of $\varphi$ that there is a function $f$ as above for which either $\{\varphi +f\geq h\}\cap U_N$ does not contain a cluster of diameter at least $N/5$ or \eqref{eq:connection} happens. However, $\pazocal{C}_1$ is contained in $\{\varphi +f\geq h\}\cap U_N$ and has diameter at least $N/5$, which implies that \eqref{eq:connection} happens.
	
	Fix now a box $B=B_{L^2}(z)\in L^2\Z^d$ that intersects $\pazocal{C}_2$, such that $D=D_{L^2}(z)$ is contained in $B_N$. Notice that 
	\begin{equation}\label{eq:max}
		\varphi_x+\max_{u\in D} f(u)\geq \varphi_x+f(x) \geq h \text{ for } x\in \pazocal{C}_2\cap D
	\end{equation}
	and 
	\begin{equation}\label{eq:min}
		\varphi_x+\min_{u\in D} f(u)\leq \varphi_x+f(x) < h \text{ for } x\in \partial^{out} \pazocal{C}_2\cap D.
	\end{equation} 
	Since $f$ is harmonic in $D_N$ and $|f(x)|< \varepsilon+\delta$ for every $x\in D_N$, we have that $|f(x)-f(y)|\leq C'(\varepsilon+\delta)/N$ for neighbouring $x$ and $y$ in $B_N$ by \eqref{eq:gradient}. Since $D$ has diameter at most $7dL^2$, we conclude that 
	$$\max_{u\in D} f(u)-\min_{u\in D} f(u)\leq 7dC'(\varepsilon+\delta) L^2/N\leq 7dC'(\varepsilon+\delta) L^{-\alpha}.$$ Consider the smallest $k\in \{-\left\lceil (\varepsilon+\delta) L^{\alpha}\right\rceil,\ldots,\left\lceil (\varepsilon+\delta) L^{\alpha}\right\rceil\}$ such that $\max_{u\in D}f(u)\leq kL^{-\alpha}$. Then $\max_{u\in D}f(u)> (k-1)L^{-\alpha}$, hence 
	\begin{equation}\label{eq:min_bound}
		\min_{u\in D}f(u)\geq \max_{u\in D}f(u)-7dC'(\varepsilon+\delta) L^{-\alpha}>(k-1-7dC'(\varepsilon+\delta))L^{-\alpha}.
	\end{equation} 
	We can now deduce from \eqref{eq:max} and \eqref{eq:min} that
	\begin{equation}\label{eq:c2}
		\varphi \geq h-kL^{-\alpha} \text{ on } \pazocal{C}_2\cap D
	\end{equation}
	and 
	\begin{equation}\label{eq:bc2}
		\varphi < h-rL^{-\alpha} \text{ on } \partial^{out} \pazocal{C}_2\cap D,
	\end{equation}
	where $r=k-1-7dC'(\varepsilon+\delta)$.
	
	Now as $\text{Conf}(h,\varepsilon,N)$ does not happen and \eqref{eq:c2} holds, for all but at most $L-1$ vertices $x$ of $\pazocal{C}_2\cap D$ we have $\varphi_x\geq h-rL^{-\alpha}$. We claim that $\{\varphi\geq h-rL^{-\alpha}\}\cap \pazocal{C}_2\cap D$ contains a cluster of diameter at least $L$. Indeed, notice that $\pazocal{C}_2\cap D$ contains a connected set of diameter at least $3L^2$ because the graph distance between $B$ and $\partial D$ is $3L^2$. Consider a path $\gamma$ in $\pazocal{C}_2\cap D$ connecting two vertices $u$ and $v$ with graph distance $d(u,v)\geq 3L^2$. Then we have that 
	$$3L^2\leq \sum_{i=1}^{j-1} d(x_i,x_{i+1}),$$
	where $x_0=u,x_j=v$ and $x_1,\ldots,x_{j-1}$ are the vertices of $\gamma$ in between $u$ and $v$ such that $h-kL^{-\alpha}\leq \varphi_{x_i}<h-rL^{-\alpha}$, ordered in turn of appearance in $\gamma$ as we move from $u$ to $v$. As $\gamma$ can have at most $L-1$ vertices $x$ such that $h-kL^{-\alpha}\leq \varphi_x<h-rL^{-\alpha}$, we can deduce that $j\leq L$, hence for some $i\in\{0,1,\ldots,j-1\}$ we have that $d(x_i,x_{i+1})\geq \frac{3L^2}{j}\geq 3L$. The subpath $\gamma'$ of $\gamma$ in between $x_i$ and $x_{i+1}$ has thus diameter at least $3L-2\geq L$ and $\varphi_x \geq h-rL^{-\alpha}$ for every $x\in \gamma'$. Consider the cluster $\pazocal{C}_3$ of $\{\varphi \geq h-rL^{-\alpha}\}\cap D$ containing $\gamma'$. We will show that $\pazocal{C}_2$ contains $\pazocal{C}_3$, which proves the claim. To this end, recall that $\pazocal{C}_2$ is a cluster of $\{\varphi+f\geq h\}\cap U_N$. Notice that for every $x\in D$, $\varphi_x+f(x)\geq \varphi_x+\min_{u\in D}f(u)>\varphi_x+rL^{-\alpha}$ by \eqref{eq:min_bound}, hence 
	\begin{equation}\label{eq:cont}
		\{\varphi \geq h-rL^{-\alpha}\}\cap D\subset \{\varphi+f\geq h\}\cap D.
	\end{equation} 
	Since $\pazocal{C}_2$ and $\pazocal{C}_3$ overlap at $\gamma'$, we deduce that $\pazocal{C}_2$ contains $\pazocal{C}_3$.
	
	Consider a box $B_L(w)\in L\Z^d$ lying in $B$ that intersects $\pazocal{C}_3$. We will show that $B_L(w)$ is $(\varphi,h-rL^{-\alpha})$-bad, which implies that $\text{NSLU}(h,\varepsilon+\delta,N)$ happens, as desired. Recalling the definition of a very dense cluster, we see that $B_L(w)$ intersects $\pazocal{C}_1$ as well. Notice that both $\pazocal{C}_3\cap U_L(w)$ and $\pazocal{C}_1\cap U_L(w)$ contain a cluster of diameter at least $L/5$ because both $\pazocal{C}_3$ and $\pazocal{C}_1$ have diameter at least $L/5$. On the other hand, $\pazocal{C}_3$ is not connected to $\pazocal{C}_1$ in $\{\varphi\geq f+h\}\cap D_L(w)$ by \eqref{eq:connection} and the fact that $\pazocal{C}_3\subset \pazocal{C}_2$. Using \eqref{eq:cont}, we can deduce that $\pazocal{C}_3$ is also not connected to $\pazocal{C}_1$ in $\{\varphi \geq h-rL^{-\alpha}\}\cap D_L(w)$. Thus $B_L(w)$ is $(\varphi,h-rL^{-\alpha})$-bad.
\end{proof}

\begin{proof}[Proof of \Lr{lemma:NSLU}]
	Let us start by constructing a very dense cluster. By increasing the value of $N$, if necessary, we can assume that for every $B=B_L(z)\in L\Z^d$ contained in $B_N$, $D=D_L(z)$ is contained in $U_N$. It is not hard to see that if all boxes of $L\Z^d$ that are contained in $B_N$ are $(\varphi,h+\varepsilon)$-good, then $\{\varphi\geq h+\varepsilon\}\cap U_N$ contains a cluster $\pazocal{C}$ such that $\pazocal{C}\cap B$ contains a cluster of diameter at least $L/5$ for every $B\in L\Z^d$ lying in $B_N$. This is because for every pair of neighbouring boxes $B$ and $B'$ contained in $B_N$, both $\{\varphi\geq h+\varepsilon\}\cap B$ and $\{\varphi\geq h+\varepsilon\}\cap B'$ contain a cluster of diameter at least $L/5$, and these two clusters are connected in $\{\varphi\geq h+\varepsilon\}\cap D$. Increasing the value of $N$ even further ensures that $\pazocal{C}$ has diameter at least $N/5$, i.e.~it is a very dense cluster.
	
	If for every $k\in\{-\left\lceil  \varepsilon L^{\alpha}\right\rceil,\ldots, 
	\left\lceil \varepsilon L^{\alpha} \right\rceil\}$, all boxes of $L\Z^d$ contained in $B_N$ are $(\varphi,h-rL^{-\alpha})$-good, then we have strong local uniqueness. Since there are at most $CM^d$ boxes of $L\Z^d$ contained in $B_N$ and we are considering $2\left\lceil  \varepsilon L^{\alpha}\right\rceil +2$ different level-sets (with the $h+\varepsilon$ level-set included), we can apply \eqref{eq:sharpness_sup} to obtain that
	$$\Pr[\text{NSLU}(h,\varepsilon,N)]\leq (2\left\lceil  \varepsilon L^{\alpha}\right\rceil +2)CM^d e^{-cL^{\rho}}\leq e^{-c'N^{\rho'}}$$
	for some constants $c'=c'(h',d)>0$ and $\rho'=\rho'(d)>0$, as desired.
\end{proof}

\begin{proof}[Proof of \Lr{lemma:small}]
	We will show that the probability of $\text{Conf}(h,\varepsilon,B)$ decays stretched exponentially for every $B\in L^2\Z^d$. Then the desired result will follow from a union bound over all $B\in L^2\Z^d$ lying in $B_N$ and the fact that there are polynomially many choices for $B$.
	
	In order to prove the aforementioned result, consider a subset $S$ of $D$ of cardinality $L$ and an integer $k\in\{-\left\lceil  \varepsilon L^{\alpha}\right\rceil,\ldots, 
	\left\lceil \varepsilon L^{\alpha} \right\rceil\}$. We will estimate the probability that $h-kL^{-\alpha}\leq \varphi_x < h-rL^{-\alpha}$ for all $x\in S$ and then apply a union bound over all possible $S$ and $k$. Let us set $h_1\coloneqq h-kL^{-\alpha}$ and $h_2\coloneqq h-rL^{-\alpha}$. Choose a subset $S'$ of $S$ such that for every $x,y\in S'$ we have $d(x,y)\geq 2$, and $S'$ is a maximal subset of $S$ with respect to this property. Then $|S'|\geq \frac{L}{2d+1}$. Now conditioning on $\varphi_y$ for $y\in\Z^d \setminus S'$, we obtain
	\begin{gather*}
		\Pr[h_1\leq \varphi_x < h_2, \forall x\in S']=\Ex\left[\Pr\left[h_1\leq \varphi_x < h_2, \forall x\in S'\mid \sigma(\varphi_y, y\in\Z^d \setminus S')\right]\right]=\\ \Ex\left[\prod_{x\in S'} \Pr\left[h_1\leq \varphi_x<h_2\mid \sigma(\varphi_y, y\in\Z^d \setminus S') \right]\right]\leq (h_2-h_1)^{|S'|}\leq (h_2-h_1)^{\frac{L}{2d+1}}.
	\end{gather*}
	For the second equality, we used that conditionally on all $\varphi_y$, $y\notin S'$, the random variables $\varphi_x$, $x\in S'$ are pairwise independent. For the first inequality, we used that $$\Pr[h_1\leq \varphi_x<h_2\mid \sigma(\varphi_y, y\in\Z^d \setminus S')]\leq h_2-h_1$$ which follows from the fact that conditionally on $\sigma(\varphi_y, y\in\Z^d \setminus S')$, $\varphi_x$ is a normal random variable with variance $1$ (the value of the mean is not important), hence its probability density function is bounded by $1$.
	
	On the other hand, $D$ contains $(7L^2)^d$ vertices, hence there are at most $(7L^2)^{dL}$ possible subsets of $D_{L^2}$ of cardinality $L$. A union bound over the $2\left\lceil \varepsilon L^{\alpha} \right\rceil +1$ possible values of $k$ and the subsets of $D_{L^2}$ of cardinality $L$ implies that $$\Pr[\text{Conf}(h,\varepsilon,D_{L^2})]\leq \left(2\left\lceil \varepsilon L^{\alpha} \right\rceil +1 \right) (7L^2)^{dL} (h_2-h_1)^{\frac{L}{2d+1}}.$$ By our choice of $\alpha$,
	\begin{align*}
		\left(2\left\lceil \varepsilon L^{\alpha} \right\rceil +1 \right)(7L^2)^{dL}(h_2-h_1)^{\frac{L}{2d+1}}&=\exp\left\{2dL\log(L)+\frac{-\alpha}{2d+1}L\log(L)++O(L)\right\}\\
		&= \exp\left\{-L\log(L)+O(L)\right\}.
	\end{align*}
	This completes the proof.
\end{proof}

\section{Decay of very-badness}\label{sec:decay_verybad}

In this section, we will prove Proposition~\ref{prop:psi-very-bad}. First, we need to express the event that $B$ is $(\psi,h,\varepsilon)$-very-bad in terms of $\varphi$ and $\xi$. 

We say that a box $B$ is $(\varphi,h,\varepsilon)$-\emph{very-good} if for every function $g:\overline{D}\rightarrow \R$ which is harmonic in $D$ and $|g(x)|<\varepsilon$ for all $x\in D$, the following happen: 
\begin{itemize}
	\item for every $B'$ which is either $B$ or some neighbour of $B$, $\{\varphi+g\geq h\}\cap B'$ contains a dense cluster,
	\item for every neighbour $B''$ of $B$ and every pair of dense clusters of $\{\varphi+g\geq h\}\cap B$ and $\{\varphi+g\geq h\}\cap B''$, respectively, there is a path in $\{\varphi+g\geq h\}\cap D$ visiting both dense clusters.
\end{itemize} 
If $B$ is not $(\varphi,h,\varepsilon)$-very-good, we will call it $(\varphi,h,\varepsilon)$-\emph{very-bad}. It is not hard to see that if $B$ is $(\psi,h,\varepsilon)$-very-bad and $(\xi,\delta)$-good for some $\delta>0$, then it is $(\varphi,h,\varepsilon+\delta)$-very-bad.

We shall now introduce another event that will be used to handle the non-uniqueness of a dense cluster. We define $H(h,\varepsilon,B)$ to be the event that there are
\begin{itemize}
	\item a function $g:\overline{D}\rightarrow \R$ which is harmonic in $D$ and $|g(x)|<\varepsilon$ for all $x\in D$, and
	\item a pair $\pazocal{C}_1,\pazocal{C}_2$ of clusters of $\{\varphi+g\geq h\}\cap U$ of diameter at least $L/5$,
\end{itemize}
for which there is no path in $\{\varphi+g\geq h\}\cap D$ connecting $\pazocal{C}_1$ with $\pazocal{C}_2$. It is not hard to see that if $H(h,\varepsilon,B)$ happens and $B$ is $(\xi,\delta)$-good for some $\delta>0$, then $B$ is $(\psi,h,\varepsilon+\delta)$-bad.

Recall that the definition of a dense cluster involves considering the boxes of $L_0\Z^d$ that are contained in $B_L$. In order to construct a dense cluster, we will need to work with the columns of this collection of $L_0$-boxes. To define them precisely, let $\{e_1,e_2,\ldots,e_d\}$ be the standard basis of $\Z^d$. Given a collection $\pazocal{F}$ of boxes of $R\Z^d$ for some $R\geq 1$, the columns of $\pazocal{F}$ parallel to $e_i$, $i\in\{1,2,\ldots,d\}$ are defined as follows. For every sequence of integers $(y_j)_{j=1,j\neq i}^d$, the set of boxes $B_R(z)\in \pazocal{F}$, $z=(z_1,z_2,\ldots,z_d)$ with $z_j=y_j$ for every $j\neq i$, will be called a \emph{column} of $\pazocal{F}$ parallel to $e_i$.  

We will now prove Proposition~\ref{prop:psi-very-bad}.

\begin{proof}[Proof of Proposition~\ref{prop:psi-very-bad}]
	Notice that if $B_L$ is $(\psi,h,\varepsilon)$-very-bad and $(\xi,\delta)$-good for some $\delta>0$, then it is $(\varphi,h,\varepsilon+\delta)$-very-bad. Applying this observation for $\delta=(h_*-h'-\varepsilon)/2$ and using \Lr{lem:multi_Szn} to handle the case that $B$ is $(\xi,\delta)$-good, we see that (after redefining $\varepsilon$) it suffices to prove that for every $h'<h_*$ and $0<\varepsilon<h_*-h'$ there is a constant $c=c(h',\varepsilon,d)>0$ such that for every $h\leq h'$
	
	$$\Pr[\text{$B_L$ is $(\varphi,h,\varepsilon)$-very-bad}]\leq e^{-c L^{d-2}}.$$
	
	We will first focus on the existence of a dense cluster. Recall that $L_0=L/M$, where $M=\left\lfloor L^{\frac{d-2}{d-1}}/\log(L)\right\rfloor$. Consider the boxes of $L_0\Z^d$ contained in $U_L$ and notice that they form a partition of $U_L$. We will show that when only a few columns of this partition contain a $(\varphi,h+\varepsilon)$-bad box, $\{\varphi\geq h+\varepsilon\}\cap B'_L$ contains a dense cluster, where $B'_L$ is either $B_L$ or a neighbouring box of $B_L$. The latter easily implies that $\{\varphi+g\geq h\}\cap B'_L$ contains a dense cluster for every function $g:\overline{D_L}\rightarrow \R$ which is harmonic in $D_L$ and $|g(x)|<\varepsilon$ for all $x\in D_L$. Then we will proceed to show that the probability of having many columns that contain a $(\varphi,h+\varepsilon)$-bad box decays exponentially in $L^{d-2}$.
	
	Among the columns of the partition of $U_L$ that are parallel to $e_i$, $i=1,2,\ldots,d$, consider those that contain a box which is $(\varphi,h+\varepsilon)$-bad. We let $\Phi_i$ be the event that there are at least $\frac{M^{d-1}}{10(2d-1)!}$ such columns. When the event $\bigcup_{i=1}^d \Phi_i$ does not happen, we will show that $\{\varphi\geq h+\varepsilon\}\cap B'_L$ contains a dense cluster. To this end, since the dense cluster needs to lie in $B'_L$, we need to restrict to the collection of boxes $B_{L_0}(z)$ such that $D_{L_0}(z)$ is contained in $B'_L$. Let us assume that $L$ is large enough so that this collection is non-empty. This collection forms a partition of a smaller box $B''_L$ that is contained in $B'_L$. Notice that the number of boxes in each column of $B''_L$ is $M-6$. Let $\Gamma$ be the set of boxes of the partition of $B''_L$ that are $(\varphi,h+\varepsilon)$-good. Then for every $e_i$, $\Gamma$ contains at least $$(M-6)^{d-1}-\frac{M^{d-1}}{10(2d-1)!}\geq \left(1-\frac{1}{5(2d-1)!}\right)(M-6)^{d-1}$$ columns parallel to $e_i$, provided that $L$ is large enough so that $$\frac{M^{d-1}}{2}\leq (M-6)^{d-1}.$$ By \Lr{connected} below, a connected component $\pazocal{F}$ of $\Gamma$ contains at least $\frac{4}{5}(M-6)^d$ boxes. Increasing the value of $L$, if necessary, we can assume that $\frac{4}{5}(M-6)^d\geq \frac{3}{4}M^d$, so that $\pazocal{F}$ contains at least $\frac{3}{4}M^d$ boxes. For each pair of neighbouring boxes $B=B_{L_0}(z)$, $B'$ in $\pazocal{F}$, both $\{\varphi\geq h+\varepsilon\}\cap B$ and $\{\varphi\geq h+\varepsilon\}\cap B'$ contain a cluster of diameter at least $L_0/5$, hence there is a path in $\{\varphi\geq h+\varepsilon\}\cap D$ visiting both clusters, where $D=D_{L_0}(z)$. By combining all these clusters, we obtain that $\{\varphi\geq h+\varepsilon\}\cap B'_L$ contains a cluster $\pazocal{C}$ visiting all boxes of $\pazocal{F}$. 
	
	To show that $\pazocal{C}$ is dense, it remains to estimate its diameter. Since $\pazocal{F}$ contains at least $\frac{3}{4}M^d$ boxes, it must intersect a column $\mathrm{Col}(\pazocal{F})$ of $B''_L$ contained entirely in $\Gamma$. Since $\pazocal{F}$ is a connected component of $\Gamma$, it must contain $\mathrm{Col}(\pazocal{F})$. In other words, $\pazocal{C}$ contains a vertex from the two extremal boxes of $\mathrm{Col}(\pazocal{F})$, which implies that $\pazocal{C}$ has diameter at least $(M-8)L_0$. We have that $(M-8)L_0=L-8L_0\geq L/5$, provided that $L$ is large enough. Thus $\pazocal{C}$ is a dense cluster.
	
	We will now estimate $\Pr[\bigcup_{i=1}^d \Phi_i]$. To this end, let $i=1,2,\ldots,d$ and $S$ be a set of $\frac{M^{d-1}}{10(2d-1)!}$ boxes that lie in different columns of $U_L$ parallel to $e_i$. We will first count the possibilities for $S$ and then estimate the probability that for a fixed $S$ as above, all its boxes are $(\varphi,h+\varepsilon)$-bad. Notice that there are $A^{d-1}$ columns parallel to $e_i$, where $A=3M$, and each column contains $A$ boxes. Hence there are at most $$2^{A^{d-1}}A^{A^{d-1}}=\exp\left\{A^{d-1}\log(2A)\right\}\leq\exp\left\{C\frac{L^{d-2}}{\log^{d-2}(L)}\right\}$$ possibilities for $S$, since we can construct $S$ by first choosing a set of columns and then picking a box from each column of this set. 
	
	Moving on to the probabilistic estimate, let $S'$ be a subset of $S$ which is well-separated and is maximal with respect to this property. Then it is not hard to see that $|S'|\geq 201^{-d} |S|$. Let $\varepsilon_0=(h_*-h'-\varepsilon)/2$. We will now consider two cases. Either at least $|S'|/2$ boxes of $S'$ are $(\xi,\varepsilon_0)$-good or at least $|S'|/2$ boxes of $S'$ are $(\xi,\varepsilon_0)$-bad. In the first case, because we have assumed that all boxes of $S$ are $(\varphi,h+\varepsilon)$-bad, we can deduce that at least $|S'|/2$ boxes of $S'$ are $(\psi,h+\varepsilon,\varepsilon_0)$-bad. Applying Proposition~\ref{prop:psi-bad} and using a union bound over the subsets of $S'$ we obtain
	$$\Pr[\text{ at least $|S'|/2$ boxes of $S'$ are $(\psi,h+\varepsilon,\varepsilon_0)$-bad]}\leq 2^{|S'|}\exp\{-c_1|S'|L_0^{\rho_1}/2\}\leq \exp\{-c_2L^{d-2}\},$$
	where in the last inequality we used that 
	\begin{equation}\label{eq:small_size}
		|S'|\leq M^{d-1} \leq \frac{L^{d-2}}{\log^{d-2}(L)}
	\end{equation}
	and 
	$$|S'|L_0^{\rho_1}\geq c_3 \frac{L^{d-2}L_0^{\rho_1}}{\log^{d-2}(L)}.$$
	On the other hand, if the second case holds, we can argue as follows. Let $T$ be the set of boxes of $S'$ that are $(\xi,\varepsilon_0)$-bad. Applying \Lr{lem:col_cap} below, we see that $\cp(\Sigma(T))\geq rL^{d-2}$ for some constant $r>0$. We shall now apply \Lr{lem:multi_Szn} and for this reason we need to check that $|T|\leq \delta rL^{d-2}$, where $\delta$ is the constant of \Lr{lem:multi_Szn}. This inequality follows from \eqref{eq:small_size} by choosing $L$ to be large enough.
	Hence a union bound over the subsets of $S'$ implies that 
	$$\Pr[\text{ at least $|S'|/2$ boxes of $S'$ are $(\xi,\varepsilon_0)$-bad]}\leq 2^{|S'|}\exp\{-cr\varepsilon_0^2 L^{d-2}\}\leq \exp\{-c_4 L^{d-2}\}.$$
	Overall, we obtain that 
	$$\Pr[\cup_{i=1}^d \Phi_i]\leq \exp\left\{C\frac{L^{d-2}}{\log^{d-2}(L)}\right\}\left(\exp\{-c_2L^{d-2}\}+\exp\{-c_4L^{d-2}\}\right)\leq \exp\{-c_5L^{d-2}\}.$$

	Let $B'_L$ be a neighbouring box of $B_L$. We shall now consider the event that for some function $g:\overline{D}\rightarrow \R$ which is harmonic in $D$ and $|g(x)|<\varepsilon$ for all $x\in D$, and a pair $\pazocal{C}_1$, $\pazocal{C}_2$ of dense clusters of $\{\varphi+g\geq h\}\cap B_L$ and $\{\varphi+g\geq h\}\cap B'_L$, respectively, there is no path in $\{\varphi+g\geq h\}\cap D_L$ connecting $\pazocal{C}_1$ to $\pazocal{C}_2$. Let $i=1,2,\ldots,d$ be such that $B'_L=B_L\pm Le_i$. Notice that for every $L_0$-box $B$ intersecting $\pazocal{C}_j$, $j=1,2$, $\pazocal{C}_j\cap U$ contains a cluster of diameter at least $L_0/5$. Hence each column of $B_L\cup B'_L$ parallel to $e_i$ that intersects both $\pazocal{C}_1$ and $\pazocal{C}_2$ must contain a box $B\in L_0\Z^d$ such that $H(h,\varepsilon,B)$ happens, since otherwise, $\pazocal{C}_1$ and $\pazocal{C}_2$ are connected in $\{\varphi+g\geq h\}\cap D_L$. We will show that many columns are intersected by both clusters. Indeed, it follows from the definition of a dense cluster that each of $\pazocal{C}_1$ and $\pazocal{C}_2$ intersects at least $3M^{d-1}/4$ columns parallel to $e_i$. In particular, at least $M^{d-1}/2$ columns parallel to $e_i$ are intersected by both $\pazocal{C}_1$ and $\pazocal{C}_2$, and all of them contain a box $B\in L_0\Z^d$ such that $H(h,\varepsilon,B)$ happens. 
	
	To estimate the probability of the event that at least $M^{d-1}/2$ columns  contain a box $B\in L_0\Z^d$ such that $H(h,\varepsilon,B)$ happens, we consider two cases. Either at least $M^{d-1}/4$ columns contain a $(\xi,\varepsilon_0)$-bad box or at least $M^{d-1}/4$ columns contain a box $B$ such that $H(h,\varepsilon,B)$ happens and $B$ is $(\xi,\varepsilon_0)$-good. When $H(h,\varepsilon,B)$ happens and $B$ is $(\xi,\varepsilon_0)$-good, $B$ is $(\psi,h,\varepsilon+\varepsilon_0)$-bad. In both cases, we can argue as above to obtain the desired decay.
\end{proof}

We will now prove the two lemmas mentioned above.

\begin{lemma}\label{connected}
	Let $d\geq 2$ and $0<x<\frac{1}{(2d-1)!}$. Consider a subset $\Gamma$ of $B_L$ such that for every direction $e_i$, $\Gamma$ contains at least $(1-x)L^{d-1}$ columns of $B_L$ parallel to $e_i$. Then $\Gamma$ contains a connected set of size at least $(1-(2d-1)!x)L^d$.
\end{lemma}
\begin{proof}
	We will prove inductively on the dimension that the statement of the lemma holds for all $\Gamma$ (as in the statement) and all $0<x<\frac{1}{(2d-1)!}$. 
	
	For $d=2$, the statement holds because any pair of vertical and horizontal columns shares a common vertex. Let us assume that it holds for some $d\geq 2$. We will prove it for $d+1$. Consider a subset $\Gamma$ of the $(d+1)$-dimensional box $B_L$ as in the statement of the lemma. For each $i=2,3,\ldots,d+1$, let $m_i$ be the number of $k\in \{0,1\ldots,L-1\}$ such that $\Gamma\cap \left(\{k\}\times [0,L)^d\right)$ contains at most $(1-y)L^{d-1}$ columns of the $d$-dimensional box $\{k\}\times [0,L)^d$ parallel to $e_i$, where $y=2dx$. Notice that $\Gamma$ contains at most $$m_i(1-y)L^{d-1}+(L-m_i)L^{d-1}=L^d-m_iyL^{d-1}$$ columns parallel to $e_i$, because for the remaining $L-m_i$ elements of the set $\{0,1\ldots,L-1\}$, $\Gamma\cap \left(\{k\}\times [0,L)^d\right)$ contains at most $L^{d-1}$ columns. Hence $L^d-m_iyL^{d-1}\geq (1-x)L^d$, which implies that $m_i\leq \frac{Lx}{y}=\frac{L}{2d}$.
	
	 Consider one of the remaining $L-\sum_{i=2}^{d+1} m_i\geq L/2$ sets $\Gamma\cap \left(\{k\}\times [0,L)^d\right)$ and notice that it satisfies the assumption of our inductive hypothesis for $x$ replaced by $y$. Hence $\Gamma\cap \left(\{k\}\times [0,L)^d\right)$ contains a connected set $S$ of size at least $(1-z)L^d$, where $z=(2d-1)!y=(2d)!x$. To find a connected set of the desired cardinality, consider some $l\neq k$, $l\in \{0,1,\ldots,L-1\}$ and notice that among the at least $(1-x)L^d$ columns of $B_L$ parallel to $e_l$ that lie in $\Gamma$, $S$ meets at least $(1-x-z)L^d\geq (1-(2d+1)!x)L^d$ of them. The union of $S$ with the columns that it meets forms a connected set of size at least $(1-(2d+1)!x)L^{d+1}$. This completes the proof.
\end{proof}

\begin{lemma}\label{lem:col_cap}
	For every $0<r<1$, there is $c=c(r,d)>0$ such that the following holds. For every $\Gamma\subset B_L$ that contains one vertex from at least $rL^{d-1}$ columns parallel to $e_1$, $\cp(\Gamma)\geq cL^{d-2}$.
\end{lemma}
\begin{proof}
	Let $F_1$ be the face of $B_L$ intersecting all columns of $B_L$ parallel to $e_1$ and let $\Gamma'$ be the projection of $\Gamma$ to $F_1$. We claim that $\cp(\Gamma)\geq t \cp(\Gamma')$ for some constant $t=t(d)>0$. Indeed, recall the variational characterization of the capacity \eqref{eq:variational_meas}. Let $\nu'$ be the probability measure supported on $\Gamma'$ such that $\cp(\Gamma')=E(\nu')^{-1}$ and define $\nu$ to be the probability measure supported on $\Gamma$ such that $\nu(x)=\nu'(x')$, where $x'$ is the projection of $x$ to $F_1$. Then $\cp(\Gamma)\geq  E(\nu)^{-1}$. Notice that by projecting $\Gamma$ onto $F_1$, the distance between its vertices decreases. Since the Green's function $g(x,y)$ is asymptotically decreasing in the distance $\Vert x-y\Vert$, we have $E(\nu')\geq t(d)E(\nu)$ and the claim follows.
	
	We will now lower bound the capacity of $\Gamma'$ by applying \eqref{eq:cap_bound}. To this end, notice that $\Gamma'$ contains at least $rL^{d-1}$ vertices and consider some vertex $x\in \Gamma'$. Since the number of vertices in $F_1$ that are at $\Vert \cdot \Vert_{\infty}$-distance $k$ from $x$ is of order $k^{d-2}$, it is not hard to see that there are constants $t_1=t_1(d,r)>0$ and $t_2=t_2(d,r)>0$ such that for at least $t_1L$ values of $k\in\{0,1,\ldots,L-1\}$, $\Gamma'$ contains at least $t_2 k^{d-2}$ vertices at distance $k$ from $x$. The desired lower bound on $\cp(\Gamma')$ follows now from \eqref{eq:cap_bound}.
\end{proof}


\end{document}